\documentclass[11pt, reqno, oneside]{amsart}
\usepackage[margin=1in]{geometry}
\usepackage{amsmath, amsfonts, amssymb, amsthm, bbm}
\usepackage{graphicx}
\usepackage{hyperref}
\usepackage{color}
\usepackage{caption}
\usepackage{subcaption}
\usepackage{datetime}

\newtheorem{thm}{Theorem}[section]
\newtheorem{prop}[thm]{Proposition}
\newtheorem{cor}[thm]{Corollary}

\newtheorem{note}[thm]{Note}
\newtheorem{lemma}[thm]{Lemma}
\newtheorem{defn}[thm]{Definition}

\newtheorem{preremark}[thm]{Remark}
\newenvironment{remark}{\begin{preremark}\rm}{\medskip \end{preremark}}
\numberwithin{equation}{section}

\newcommand{\norm}[1]{\left\Vert#1\right\Vert}
\newcommand{\abs}[1]{\left\vert#1\right\vert}
\newcommand{\Indicator}{{\mathbbm{1}}}

\newcommand{\Union}{\bigcup}
\newcommand{\intersect}{\cap}

\newcommand{\set}[1]{\left\{#1\right\}}
\newcommand{\R}{\mathbb R}

\newcommand{\al}{\alpha}
\newcommand{\del}{\delta}
\newcommand{\gam}{\gamma}
\newcommand{\eps}{\varepsilon}
\newcommand{\Lam}{\Lambda}
\newcommand{\lam}{\lambda}

\newcommand{\grad} {\nabla}

\newcommand{\dx} {\; \mathrm{d} x}
\newcommand{\dd} {\; \mathrm{d}}

\DeclareMathOperator*{\osc}{osc}

\newcommand{\commentout}[1]{}

\begin{document}

\title{Regularity for parabolic integro-differential equations with very irregular kernels}

\author{Russell W. Schwab}
\author{Luis Silvestre}

\address{Department of Mathematics\\
Michigan State University}
\email{rschwab@math.msu.edu}

\address{
Department of Mathematics\\
University of Chicago
}
\email{luis@math.uchicago.edu}

\date{\today}

\begin{abstract}
We prove H\"older regularity for a general class of parabolic integro-differential equations, which (strictly) includes many previous results.  We present a proof which avoids the use of a convex envelop as well as give a new covering argument which is better suited to the fractional order setting. Our main result involves a class of kernels which may contain a singular measure, may vanish at some points, and are not required to be symmetric. This new generality of integro-differential operators opens the door to further applications of the theory, including some regularization estimates for the Boltzmann equation.
\end{abstract}

\thanks{L. Silvestre was partially supported by NSF grants DMS-1254332 and DMS-1065979. }

\maketitle

\pagestyle{headings}		
\markboth{R. Schwab and L. Silvestre, Parabolic Regularity}{R. Schwab and L. Silvestre, Parabolic Regularity}

\section{Introduction}

We study the H\"older regularity for solutions of integro-differential equations of the form
\begin{equation} \label{e:main}  
	u_t + b(x,t) \cdot \grad u - \int_{\R^d} \big(u(x+h,t)-u(x,t) \big) K(x,h,t) \dd h = f(x,t).
\end{equation}
The integral may be singular at the origin and must be interpreted in the appropriate sense.
These equations now appear in many contexts. Most notably, they appear naturally in the study of stochastic processes with jumps, which traditionally has been the main motivation for their interest.  In the same way that pure jump processes contain the class of diffusions (processes with continuous paths) as particular limiting cases, the equation (\ref{e:main}) contains the usual second order parabolic equations as particular limiting cases.  This is due to the fact that the integral term becomes a second order operator $a_{ij}(x,t) \partial_{ij} u$ as the order $\al$ (to be defined below) converges to two.  We note that the simplest choice of $K$ is $K(h)=C_{d,\al}\abs{h}^{-d-\al}$, which results in the equation
\[
u_t + (-\Delta)^{\al/2}u=0,
\]
and converges to the usual heat equation $u_t-\Delta u=0$ as $\al\to2$ (recall that $(-\Delta)^{\al/2}$ is the operator whose Fourier symbol is $\abs{\xi}^\al$).

The H\"older estimates that we obtain in this article are an integro-differential version of the celebrated result by Krylov and Safonov for parabolic equations with measurable coefficients \cite{krylov1980certain}. There are in fact several versions of these H\"older estimates for integro-differential equations which were obtained in the last 10 years, and we briefly review them in Section \ref{sec:OtherResults}.  Besides the elliptic / parabolic distinction, the difference between each version of the estimates is in the level of generality in the possible choices of the kernels $K(x,h,t)$. In this article we obtain the estimates for a very generic class of kernels $K$, including nearly all previous results of this type.

The most common assumption in the literature is that for all $x$ and $t$, the kernel $K$ is comparable pointwise in terms of $h$ to the kernel for the fractional Laplacian. More precisely
\begin{equation} \label{e:pointwise-bound-for-K}  
	(2-\al) \frac {\lambda}{|h|^{d+\al}} \leq K(x,h,t) \leq (2-\al) \frac {\Lambda}{|h|^{d+\al}}.
\end{equation}
This assumption is often accompanied by the symmetry assumption: $K(x,h,t) = K(x,-h,t)$. It is important for the applications of these estimates that no regularity condition may be assumed for $K$ with respect to $x$ or $t$.

In this paper, we only assume a much weaker version of \eqref{e:pointwise-bound-for-K}.  The upper bound for $K$, in (\ref{e:pointwise-bound-for-K}), is relaxed to hold only in average when we integrate all the values of $h$ on an annulus, and it appears as assumption (A2).  Also, for our work, the lower bound in (\ref{e:pointwise-bound-for-K}) only needs to hold in a subset of values of $h$ which has positive density, given as assumption (A3). We also make an assumption, (A4), which says that the odd part of $K$ is under control if $\alpha$ is close to one. The exact conditions are listed in Section \ref{sec:KernelsAndAssume}.  We prove that solutions of (\ref{e:main}) are uniformly H\"older continuous, which we state in an informal way here and revisit more precisely in Section \ref{sec:Holder}.

\begin{thm}\label{thm:intro}
Let $u$ solve (\ref{e:main}).
Assume that for every $x \in B_1$ and $t \in [-1,0]$, the kernel $K(x,\cdot,t)$ satisfies the assumptions (A1), (A2), (A3) and (A4) in Section \ref{sec:KernelsAndAssume}. Assume also that $f$ is bounded, $b$ is bounded, and for $\al<1$ also $b\equiv 0$, then for some $\gam>0$
\[
[u]_{C^{\gam}(Q_{1/2})}\leq C(\norm{u}_{L^\infty(\R^d\times[-1,0])} + \norm{f}_{L^\infty(Q_1)}).
\]
\end{thm}
The constants $C$ and $\gamma$ depend on the constants $\mu$, $\lambda$ and $\Lambda$ in (A1)-(A4), on the dimension $d$, on a lower bound for $\alpha$ (in particular $\alpha$ can be arbitrarily close to two), and on $\|b\|_{L^\infty}$.

Our purpose in developing Theorem \ref{thm:intro} is not merely for the sake of generalization. An estimate with the level of generality given here can be used to obtain a prior estimates for the homogeneous Boltzmann equation. This is is a novel application. None of the previous H\"older estimates for integral equations are appropriate to be applied to the Boltzmann equation.

As a byproduct of our proof of Theorem \ref{thm:intro} we simplify and clarify some of the details regarding parabolic covering arguments (see the Crawling Ink Spots of Section \ref{sec:LEpsilon}) as well as present a proof which does not invoke a convex envelop.  Rather, we circumvent the often-used gradient mapping of the convex envelop by using a mapping which associates points via their correspondence through parameters in an inf-convolution, modeled on the arguments of \cite{ImbertSilvestre2013EstimatesOnlyWhenGradLargeArXiv}, originating in \cite{Cabre-1997NondivergentEllpticOnManifoldCPAM}, \cite{Savin-2007SmallPerturbationCPDE}.

In the section \ref{s:nonlinear}, we apply this result to derive the $C^{1,\alpha}$ regularity for the parabolic Isaacs equation. This is a rather standard application of H\"older estimates for equations with rough coefficients as in Theorem \ref{thm:intro}.


\subsection{Comparison with previous results and some discussion of (\ref{e:main})}\label{sec:OtherResults}

The H\"older estimates for integro-differential equations which take the form of (\ref{e:main}) are a fractional order version of the classical theorem by Krylov and Safonov \cite{krylov1980certain}. This is a fundamental result in the study of regularity properties of parabolic equations in non divergence form, and has consequences for many aspects of the subsequent PDE theory.   The classical theorem of De Giorgi, Nash and Moser concerns second order parabolic equation in \emph{divergence form}, in contrast with the theorem of Krylov and Safonov. The basic results for integro-differential equations in divergence form were developed earlier, and a small survey of this subject can be found in \cite{KassSchwa-2013RegularityNonlocalParaParma}.

The simplest case of $K$ would be $K(h)=(2-\al)\abs{h}^{-d-\al}$, and this choice gives the operator $Lu(x) = -C_{d,\al}(-\Delta)^{\al/2}u(x)$, which is a multiple of the fractional Laplacian of order $\al$ (the operator whose Fourier symbol is $\abs{\xi}^\al$).  This operator (and its inverse, the Riesz potential of order $\al$) have a long history, and have been fundamental to potential theory for about a century, see for example Landkof's book \cite{Land-72}.  In fact, the appearance of nonlocal operators similar the one in (\ref{e:main}) is in some sense generic among all linear operators which satisfy the positive global maximum principle (that the operator is non-positive whenever it is evaluated at a positive maximum of a $C^2$ function).  This has been known since the work of Courr\`ege in \cite{Courrege-1965formePrincipeMaximum}. He proved that any linear operator with the positive maximum principle must be of the form
\begin{align*}
Lu(x) = -c(x)u(x) &+ b(x)\cdot \grad u(x) + \textnormal{Tr}(A(x)D^2u(x))\\
& + \int_{\R^d} \left(u(x+h)-u(x)-\Indicator_{B_1}(h)\grad u(x)\cdot h\right)\mu(x,\dd h),
\end{align*}
where $c\geq 0$ is a function, $A\geq 0$ is a matrix, $b$ is a vector, all of $A$, $b$, $c$ are bounded, and $\mu(x,\cdot)$ is a L\'evy measure which satisfies
\[
\sup_x \int_{\R^d}\min(\abs{h}^2,1)\mu(x,\dd h) < +\infty.
\]
Heuristically from the point of view of jump-diffusion stochastic processes, $b$ records the drift, $A$ records the local covariance (or $\sqrt{A}$ is the diffusion matrix), and $\mu$ records the jumps.

The first H\"older regularity result for an equation of the form \eqref{e:main} was obtained in \cite{bass2002harnack}. In that paper, the authors consider the elliptic equation ($u$ constant in time), with symmetric kernels satisfying the pointwise bound \eqref{e:pointwise-bound-for-K} and without drift. Their proof uses probabilistic techniques involving a related Markov (pure jump) stochastic process. Other results using probabilistic techniques were \cite{BaKa-05Holder} and \cite{song2004harnack}, where different assumptions on the kernels are considered. The first purely analytical proof was given in \cite{silvestre2006holder}.  This first generation of results are all only elliptic problems. They are not \emph{robust} in the sense that as the order approaches two, the constants in the estimates blow up (hence not recovering the known second order results).  Furthermore, they all require a pointwise bound below for the kernels as in \eqref{e:pointwise-bound-for-K}.

The first robust H\"older estimate for the elliptic problem was obtained in \cite{caffarelli2009regularity}, which means that the estimate proved in \cite{caffarelli2009regularity} has constants that do not blow up as the order $\alpha$ of the equation goes to two. In that sense, \cite{caffarelli2009regularity} is the first true generalization of the theorem of Krylov and Safonov. It was the first of the series of papers \cite{caffarelli2009regularity}, \cite{caffarelli2011regularity} and \cite{caff-silv-EK} recreating the regularity theory for fully nonlinear elliptic equations in the nonlocal setting. As above, these results are only for the elliptic problem, and they require symmetric kernels which satisfy the pointwise assumption \eqref{e:pointwise-bound-for-K}.

The first estimate for parabolic integro-differential equations, \emph{in non-divergence form}, appeared, to the best of our knowledge, in \cite{silvestre2011differentiability} (the divergence case had some earlier results such as \cite{bass2002transition}, \cite{ChKu03}). In this case the kernels are symmetric and satisfy \eqref{e:pointwise-bound-for-K} with $\alpha=1$. The focus of \cite{silvestre2011differentiability} is on the interaction between the integro-differential part and the drift term. The proof can easily be extended to arbitrary values of $\alpha$, but the estimate is not robust (it blows up as $\alpha \to 2$), and the details of this proof are explained in the lecture notes by one of the authors \cite{silvestre-lecturenotes}. It is even possible to extend this proof to kernels which satisfy the upper bound in average like in our assumption (A2) below (see \cite{silvestreICM}). However, the estimates are not robust, and the lower bound in \eqref{e:pointwise-bound-for-K} is required.

The first robust estimate for parabolic equations appeared in \cite{lara2014regularity}, which is a parabolic version of the result in \cite{caffarelli2009regularity}. The kernels are required to be symmetric and to satisfy the two pointwise inequalities \eqref{e:pointwise-bound-for-K} as an assumption.

Elliptic integro-differential equations with non-symmetric kernels are studied in the articles \cite{Chan-2012NonlocalDriftArxiv} and \cite{chang2012regularity}.  There, the kernels are decomposed into the sum of their even (symmetric) and odd parts. The symmetric part is assumed to satisfy \eqref{e:pointwise-bound-for-K}, and there are appropriate assumptions on the odd part so that the symmetric part of the equations controls the odd part.  This effectively makes the contribution to the equation from the odd part of the kernel a lower order term.

The only articles where the lower bound in the kernels \eqref{e:pointwise-bound-for-K} is not required to hold at all points are \cite{BjCaFi-2012NonlocalGradDepOps}, \cite{GuSc-12ABParma}, \cite{KassmannMimica-2013NondegJumpStochProApp}, and \cite{rang2013h}. These papers concern elliptic equations and the upper bound in \eqref{e:pointwise-bound-for-K} is still assumed to hold.  It is important to point out that under the conditions in \cite{BjCaFi-2012NonlocalGradDepOps} and \cite{rang2013h}, the Harnack inequality is not true. There is in fact a counterexample in \cite{BogdanSztonyk-2005HarnackStable} (also discussed in \cite{rang2013h}). The assumption in these works which was made to replace the pointwise lower bound on the kernels is more restrictive than our assumption (A3) below.

The main result in this article, see Theorems \ref{t:holder-no-drift} and \ref{t:holder-with-drift}, generalizes nearly all previous H\"older estimates (for both elliptic and parabolic equations) for integro-differential equations with rough kernels in non-divergence form. It strictly contains the H\"older regularity results in:  \cite{bass2002harnack}, \cite{BjCaFi-2012NonlocalGradDepOps}, \cite{caffarelli2009regularity}, \cite{Chan-2012NonlocalDriftArxiv}, \cite{chang2012regularity}, \cite{lara2014regularity},    \cite{GuSc-12ABParma},  and \cite{rang2013h}. There is an interesting new result given in \cite{KassmannMimica-2013IntrinsicScalingArXiv} which allows for kernels with a \emph{logarithmic} growth at the origin (among other cases), corresponding in our context to the limit $\al\to0$, and it is not contained in the result of this paper.

Our approach draws upon ideas from several previous papers. Moreover, we haven been able to simplify the ideas substantially, especially how to handle parabolic equations, and we do not follow the method in \cite{lara2014regularity}. Our method allows us to make more general assumptions on the class of possible kernels. We would like to point out that we do not make any assumption for \emph{simplicity} in this paper. Extending these results to a more singular family of kernels would require new ideas.

There are two possible directions that we did not pursue in this paper. We did not try to analyze singularities of the kernels of order more general than a power of $|h|$ as in \cite{KassmannMimica-2013IntrinsicScalingArXiv}. Also, it might be possible to extend our regularity results for equations with H\"older continuous drifts and $\alpha < 1$ as in \cite{silvestre2010holder}, although we do point out that the technique from \cite{silvestre2010holder} does not work right away with the methods in this paper.  We also point out that the results in this paper and all of the others mentioned (except for \cite{KassSchwa-2013RegularityNonlocalParaParma}), require that the L\'evy measure-- referred to above as $\mu(x,\dd h)$-- has a nontrivial absolutely continuous part, $K\dd h$, with respect to Lebesgue measure (our work allows for a measure with a density plus some singular part).  Verifying the validity of, and finding a proof for, results similar to Theorem \ref{thm:intro} in the case when $\mu$ may not have a density with respect to Lebesgue measure remains a significant open question in the integro-differential theory.

The importance of not assuming any regularity in $x$ and $t$ for the ingredients of (\ref{e:main})-- the case of so-called bounded measurable coefficients-- is for much more than simply mathematical generality.  For example, because equations such as (\ref{e:main}) often lack a ``divergence structure''-- i.e. admitting a representation as a weak formulation for functions in an energy space such as $H^{\al/2}$-- they can usually only be realized as classical solutions or as viscosity solutions (weak solutions).  (We note that uniqueness for equations related to (\ref{e:main}) is still an open question for the theory of viscosity solutions of integro-differential equations, and recent progress has been made in \cite{mouSwiech-uniqueness}.)  That means that one of the few tools available for compactness arguments involving families of solutions are those provided in the space of continuous functions via Theorem \ref{thm:intro}.  This is relevant for both the possibility of proving the existence of classical solutions as well as for analyzing fully nonlinear equations in a way that doesn't depend on the regularity of the coefficients.  Indeed, both situations can be viewed as morally equivalent to studying linear equations with bounded measurable coefficients.  For studying regularity of translation invariant equations this arises by effectively differentiating the equation, which results in coefficients which depend upon the solution.  In the fully nonlinear case, many situations involve operators which are a min-max of linear operators, and so the bounded measurable linear coefficients arise from choosing the operators which achieve the min-max for the given function at each given point-- a situation in which you cannot assume any regular dependence in the $x$ variable.  Such min-max representations turn out to be somewhat generic for fully nonlinear elliptic equations and was noted in the recent work \cite[Section 4]{GuSc-2014NeumannHomogPart1arXiv}.

\subsection{Application: the homogeneous non cut-off Boltzmann equation}
In this section, we briefly explain an important application of our main result which is not possible to obtain with any of the previously known estimates for integro-differential equations. This result is explained in detail in \cite{luis-in-progress}.

The Boltzmann equation is a well known integral equation that models the evolution of the density of particles in dilute gases. In the space homogeneous case, the equation is
\begin{equation} \label{e:Boltzmann}  f_t = Q(f,f) := \int_{\R^n} \int_{\partial B_1} (f(v',t) f(v'_\star,t) - f(v_\star,t) f(v,t)) B(|v-v_\star|,\theta) \dd \sigma \dd v_\star.
\end{equation}
Here $v'$, $v'_\star$ and $\theta$ are defined by the relations
\begin{align*}
r &= |v_\ast-v| = |v_\ast'-v'|, \\
\cos \theta &= \sigma \cdot \frac{v_\ast-v}{|v_\ast-v|}, \\
v' &= \frac{v+v_\ast} 2 + \frac{r} 2 \sigma,\\
v_\ast' &= \frac{v+v_\ast} 2 - \frac{r} 2 \sigma.
\end{align*}

There are several modeling choices for the \emph{cross section} function $B$. From some physical considerations, it makes sense to consider $B(r,\theta) \approx r^\gamma |\theta|^{n-1+\alpha}$, with $\gamma > -n$ and $\al \in (0,2)$. Note that this cross section $B$ is never integrable with respect to the variable $\sigma \in \partial B_1$. In order to avoid this difficulty, sometimes a (non physical) cross section is used which is integrable. This assumption is known as \emph{Grad's cut-off assumption}.

Until the middle of the 1990's, most works on the Boltzmann equation used Grad's cut-off assumption. The non cut-off case, despite of its relevance for physical applications, was not studied so much due to its analytical complexity.
An important result that resulted in a better understanding of the non cut-off case came with the paper of Alexandre-Desvillettes-Villani-Wennberg \cite{boltzmann}, in which they obtained a lower bound on the entropy dissipation in terms of the Sobolev norm $\|f\|_{loc}^{\alpha/2}$. All regularity results for the non cut-off case which came afterwards are based on a coercivity estimate which is a small variation of this entropy dissipation argument. So far, this was the only regularization mechanism which was known for the Boltzmann equation.

It turns our that we can split the right hand side of the Boltzmann equation, (\ref{e:Boltzmann}), in two terms. The first one is an integro-differential operator, and the second is a lower order term.
{
\allowdisplaybreaks
\begin{align*}
f_t &= Q_1(f,f) + Q_2(f,f), \\
&:=\int_{\R^n} \int_{\partial B_1} f(v'_\star,t) (f(v',t) - f(v,t)) B(|v-v_\star|,\theta) \dd \sigma \dd v_\star \\ 
&\phantom{=} \qquad + f(v,t) \int_{\R^n} \int_{\partial B_1} (f(v'_\star,t) - f(v_\star,t)) B(|v-v_\star|,\theta) \dd \sigma \dd v_\star, \\
&= \int_{\R^n} (f(v',t) - f(v,t)) K_f(v,v',t) \dd v' + c f(v,t) [|v|^\gamma \ast f](v).
\end{align*} 
}

The kernel $K_f$ depends on $f$ through a complicated change of variables given using the integral identity above. If one knew that $f$ was a smooth positive function vanishing at infinity, then indeed it could be proved that $K_f(v,v',t) \approx |v-v'|^{-n-\alpha}$, and the first term would correspond to an integro-differential operator of order $\alpha$ in the usual sense satisfying \eqref{e:pointwise-bound-for-K}. Unfortunately, this is not practical for obtaining basic a priori estimates for (\ref{e:Boltzmann}). In fact, there is very little we can assume a priori from the solution $f$ to the Boltzmann equation, and it is not enough to conclude that $K_f$ satisfies \eqref{e:pointwise-bound-for-K}. Instead, all we know a priori about $f$ is given by its \emph{macroscopic quantities}: its mass (the integral of $f$), the energy (its second moment), and its entropy. The first two quantities are constant in time, whereas the third is monotone decreasing. It can be shown that $K_f$ satisfies the hypothesis (A1), (A2), (A3) and (A4) depending on these macroscopic quantities only. Therefore, the results in this article can be used to obtain a prior estimates for solutions of the homogeneous, non cut-off, Boltzmann equation, which is explained in \cite{luis-in-progress}. It is a new regularization effect for the Boltzmann equation which is not based on coercivity estimates as in \cite{boltzmann}.

Interestingly enough, the macroscopic quantities do not give much more information about $K_f$ than what our assumptions (A1), (A2) and (A3) say. The kernels $K_f$ will be symmetric, so in fact (A4) is redundant. In terms of this generalization, almost the full power of our main result is needed. The only non essential points are that the kernels can be assumed to be symmetric, and the robustness of the estimates does not necessarily play a role.

\subsection{Notation}

\begin{itemize}
\item Our space variable $x$ belongs to $\R^d$.

\item The annulus is $R_r:= B_{2r}\setminus B_r$

\item The parabolic cylinder $Q_r$ is defined as
\[ Q_r := B_r \times (-r^\alpha,0]\ \text{and with a different center}\ Q_r(x,t) = Q_r + (x,t).\]

\item The ``$\al$-growth'' class is 
\[
Growth(\al) = \{ v: \R^d\to \R\ |\ \abs{v(x)}\leq C(1+\abs{x})^{\al-\eps}\ \text{for some}\ C, \eps >0    \}
\]

\item Pointwise $C^{1,1}$
\begin{align*}
C^{1,1}(x) := \{ v&:\R^d\to\R\ |\ \exists\ M(x)\ \text{and}\ \eps\ \\
&\text{so that}\ \abs{v(x+h)-v(x)-\grad v(x)\cdot h}\leq M(x)\abs{h}^2\ \text{for}\ \abs{h}<\eps    \}
\end{align*}

\item 
\begin{align*}
C^{1,1}(\R^d) := \{ v&:\R^d\to\R\ |\ \norm{v}_{L^\infty(\R^d)}<\infty,\ \norm{\grad v}_{L^\infty(\R^d)}<\infty, \\ &\text{and}\ v\in C^{1,1}(x)\ \forall x\ \text{with}\ M(x)\ \text{is independent of}\ x     \}
\end{align*}

\item The difference operator for the different possibilities of $\al$ is
\begin{equation*}
	\del_y u(x) := 
	\begin{cases}
	u(x+y) - u(x)\ &\text{if}\ \al < 1\\	
	u(x+y) - u(x) - \Indicator_{B_1}(y) \nabla u(x) \cdot y\ &\text{if}\ \al = 1\\
	u(x+y) - u(x) - \nabla u(x) \cdot y\ &\text{if}\ \al >1
	\end{cases} 
\end{equation*}

\item The class of kernels and corresponding linear operators are
\[
\mathcal K := \{K\ :\ \R^d\to\R\ |\ K\ \text{satisfies assumptions (A1)-(A4)} \}
\]
\[
\mathcal L := \{ Lu(x) = \int_{\R^d}\del_h u(x) K(h) \dd h\ |\ K\in\mathcal K \}
\]

\end{itemize}

\noindent
We will try to stick to the following conventions for constants:
\begin{itemize}
	\item Large constants will be upper case letters, e.g. $C$, and small constants will be lower case letters, e.g. $c$.
	\item If the value of a constant is not relevant for later arguments, then we will freely use the particular letter for the constant without regard to whether or not it was used previously or will be used subsequently.
	\item If the value of a constant is relevant to later arguments (e.g. in determining values of subsequent constants), then we will label the constant with a subscript, e.g. $C_0$, $C_1$, $C_2$, etc... 
\end{itemize}

\begin{note}
The following observation is useful and applies for all values of $\alpha$: if $u(x) = \varphi(x)$ and $u \geq \varphi$ everywhere, then $\delta_h u(x) \geq \delta_h \varphi(x)$ for all $h$.  This implicitly assumes that for $\al\geq 1$ that $u$ and $\phi$ are both differentiable at $x$.
\end{note}

\section{Classes of kernels and extremal operators}
\label{sec:KernelsAndAssume}

The kernel $K(x,h,t)$ in \eqref{e:main} is not assumed to have any regularity with respect to $x$ or $t$. The best way to think about it is that for every value of $x$ and $t$ we have a kernel ($K_{x,t}(h) = K(x,\cdot,t)$) which belongs to certain class. This class of kernels is what we describe below.

\subsection{Assumptions on $K$}

For each value of $\lambda$, $\Lambda$, $\mu$ and $\alpha$, we consider the family of kernels $K: \R^d \to \R$ satisfying the following assumptions.
\begin{enumerate}
\item[(A1)] $K(h) \geq 0$ for all $h\in\R^d$.

\item[(A2)] For every $r>0$,
\begin{equation} \label{e:assumption-above}  \int_{B_{2r} \setminus B_r} K(h) \dd h \leq (2-\alpha) \Lambda r^{-\alpha}
\end{equation}

\item[(A3)] For every $r>0$, there exists a set $A_r$ such that
\begin{itemize}
\item $A_r \subset B_{2r} \setminus B_r$.
\item $A_r$ is symmetric in the sense that $A_r = -A_r$.
\item $|A_r| \geq \mu |B_{2r} \setminus B_r|$.
\item $K(h) \geq (2-\alpha) \lambda r^{-d-\alpha}$ in $A_r$.
\end{itemize}
Equivalently
\begin{equation} \label{e:assumption-below}  \abs{ \set{y \in B_{2r} \setminus B_r: K(h) \geq (2-\alpha) \lambda r^{-d-\alpha} \text{ and } K(-h) \geq (2-\alpha) \lambda r^{-d-\alpha} }} \geq \mu |B_{2r} \setminus B_r|.
\end{equation}

\item[(A4)] For all $r>0$,
\begin{equation} \label{e:assumption-odd}  \left\vert \int_{B_{2r} \setminus B_r} h K(h) \dd h \right\vert \leq \Lambda |1-\alpha| r^{1-\alpha}.
\end{equation}
\end{enumerate}


\subsection{Discussion of the assumptions}

We stress that although our kernels can be zero for large sets of $h$, their corresponding integral operators are not rightfully described as ``degenerate''.  One can draw an analogy with the second order case in the context of diffusions.  A diffusion process will satisfy uniform hitting time estimates for measurable sets of positive measure whenever the diffusion matrix is comparable to the identity from below and above.  In the context of our pure jump processes related to (\ref{e:main}), these jump processes will still satisfy such uniform hitting time estimates even though the kernels can be zero in many points (meaning that at the occurrence of any \emph{one} jump, the process will have zero probability of jumping with certain values of $h$).

The first assumption, (A1), is unavoidable if one hopes to study examples of (\ref{e:main}) which satisfy a comparison principle between sub and super solutions.

The second assumption, (A2), is mostly used to estimate an upper bound for the application of the operator, $L$, to a smooth test function. It is more general than assuming a pointwise upper bound such as was done in \cite{caffarelli2009regularity}, \cite{rang2013h} and many others. It is also slightly more general than a corresponding bound obtained by integrating on spheres as
\[ \int_{\partial B_r} K(h) \dd S(h) \leq (2-\alpha) \Lambda r^{-1-\alpha}.\]
It is however, a stronger hypothesis than 
\[ \int_{B_r} |h|^2 K(h) \dd h \leq \Lambda r^{2-\alpha}.\]
It is worth pointing out that (A2) implies
\[ 
\int_{\R^d \setminus B_r} 
K(h) \dd h \leq \frac{2^\alpha}{2^\alpha-1}(2-\alpha) \Lambda r^{-\alpha}.
\]
The first factor blows up as $\alpha \to 0$ but not as $\alpha \to 2$. In fact, the proofs of all our regularity results fail for $\alpha \leq 0$ exactly because the tails of the integrals become infinite.  The question of what happens as $\al\to0$ is interesting for the nonlocal theory, and some results are obtained in \cite{KassmannMimica-2013IntrinsicScalingArXiv} (note, there they do not use the typical normalization constant as in potential theory, where $C_{d,\al}\approx \al$ as $\al\to0$, so the limit operator is \emph{not} a multiple of the identity).
We also have
\begin{equation} \label{e:integrability-of-K} 
	 \int_{\R^d} (1\wedge |h|^2) K(h) \dd h \leq C(\alpha) \Lambda,
\end{equation}
for a constant $C(\alpha)$ which stays bounded as $\alpha \to 2$, and \eqref{e:assumption-above} can be thought of as a scale invariant, of order $\alpha$, version of \eqref{e:integrability-of-K}.

Note that the assumption (A2) does not preclude the kernel $K$ to contain a singular measure. For example, the measure given by
\[ \int_A K(h) \dd h = \int_{A \cap \{h_1=h_2=\dots=h_{d-1}=0\}} (2-\alpha) \frac \lambda {|h_n|^{1+\alpha}} \dd h_d,\]
is a valid kernel $K$ which satisfies (A2) (but not (A3)). In this case $K$ is a singular measure, but we abuse notation by writing it as if it was absolutely continuous with a density $K(h)$.

The example above corresponds to the operator
\[ -\int_{\R^d} \delta_h u(x) K(h) \dd h = (-\partial_{dd})^{\alpha/2} u.\]
As we mentioned before, this kernel satisfies the assumption (A2) but not (A3). However, the kernel of the operator
\[ -\int_{\R^d} \delta_h u(x) K(h) \dd h := (-\partial_{dd})^{\alpha/2} u(x) + (-\Delta)^{\alpha/2} u(x)\]
would satisfy both (A2) and (A3).

The third assumption, (A3), is stated in a form which does not require the kernel $K$ to be positive along some prescribed rays or cone-like sets as was done in \cite{rang2013h}. The relaxation to (A3) from previous works is important to allow for situations where the positivity set of $K$ may change from radius to radius. As mentioned above, it is equivalent to \ref{e:assumption-below}, which is the form we will actually invoke later on.

Finally, we note that the assumption, (A4), is automatic for symmetric kernels (i.e. when $K(h) = K(-h)$), since in that case the left hand side is identically zero. This assumption is made in order to control the odd part of the kernels in a fashion that does not require us to split up $L$ into two pieces involving the even and odd parts of $K$.  It is also worth pointing out that even for $\al<1$, $K$ can have some asymmetry, but it must die out as $r\to\infty$.   

There are two final facts which are important to point out.  The first one is the observation that although each $K$ may not be such that 
\[
Lu(x) = \int_{\R^d} \del_hu(x)K(h) \dd h
\]
results in an operator which is scale invariant, i.e. $Lu(r\cdot)(x) = r^\al Lu(rx)$, the \emph{family of $K$} which satisfy (A1)-(A4) is scale invariant.   The second one is that some authors have worked with assumptions where the lower bound in (\ref{e:pointwise-bound-for-K}) is only required for $\abs{h}\leq 1$.  This does not effect our overall result because we can add and subtract the term
\[
f(u;x) := (2-\al)\int_{\R^d} \del_hu(x)\Indicator_{\R^d\setminus B_1}(h)\abs{h}^{-d-\al}\dd h
\]
from the equation (\ref{e:main}).  Assuming $K$ satisfies the lower bound of (\ref{e:pointwise-bound-for-K}) only for $\abs{h}\leq 1$, this would result in an operator governed by $\tilde K(h) = K(h) + \Indicator_{\R^d\setminus B_1}(h)\abs{h}^{-d-\al}$, and now $\tilde K$ does satisfy the lower bound of (\ref{e:pointwise-bound-for-K}) for all $h$.  Furthermore, the term $f(u;\cdot)$ is controlled by $\norm{u}_{L^\infty}$ and possibly $C\abs{\grad u}$ (depending on $\al$) due to the fact that $\Indicator_{\R^d\setminus B_1}(h)\abs{h}^{-d-\al}$ is integrable, and hence these terms can be absorbed into the equation as a gradient term and bounded right hand side.  This pertains to, e.g., the results in \cite{Chan-2012NonlocalDriftArxiv}.



\subsection{Extremal operators and useful observations}

As mentioned above, $\mathcal L$ is the class of all integro-differential operators $Lu$ of the form
\[
Lu(x) = \int_{\R^d} \delta_h u(x) K(h) \dd h,
\]
where $K$ is a kernel satisfying the assumptions (A1)-(A4) specified above.  Sometimes we wish to refer to a kernel, $K$, instead of the operator, $L$, and so we also use $\mathcal K$ to denote the collection of all such kernels.  Correspondingly, we define the extremal operators $M^-_{\mathcal L}$ and $M^-_{\mathcal L}$ as in \cite{caffarelli2009regularity}:
\begin{align*}
M^+_{\mathcal L} u(x) &= \sup_{L \in \mathcal L} Lu(x), \\
M^-_{\mathcal L} u(x) &= \inf_{L \in \mathcal L} Lu(x).
\end{align*}
In order to avoid notational clutter, we omit the subscript $\mathcal L$ in the rest of the paper.  We note that when \eqref{e:main} holds for some kernel $K$ satisfying the assumptions and with a bounded $b$ and $f$, this implies also the pair of inequalities is simultaneously satisfied:
\begin{align*}
u_t + C_0 |\grad u| - M^- u &\geq -C_0, \\
u_t - C_0 |\grad u| - M^+ u &\leq C_0.
\end{align*}
The advantage of this new formulation is that it can be understood in the viscosity sense, whereas the original equation \eqref{e:main} only makes sense for classical solutions.  Unless otherwise noted, we use the terms solution, subsolution, and supersolution to to be interpreted in the viscosity sense (made precise below, in Definition \ref{d:viscosity}).  There may be instances when we need equations to hold in a classical sense, and in those cases we will explicitly mention that need.

\begin{remark}
	We emphasize that although (\ref{e:main}) allows for $K$ which are $x$-dependent, the class $\mathcal L$-- and hence the definition of $M^\pm$-- contains only those $K$ which are independent of $x$.  The reason that the desired inequalities are obtained is that $\mathcal L$ contains \emph{all possible} such $K$, and hence, for each $x$ fixed, $K(x,\cdot)\in \mathcal L$.
\end{remark}

It will be useful to know an important feature of $M^\pm$ regarding translations, rotations, and scaling.  This is an important feature to keep in mind in the sense that for any \emph{one} choice of a kernel to determine (\ref{e:main}), $K$ may not have any symmetry or scaling properties on its own.  However, it is controlled by an extremal operator which does enjoy these properties.  This is particularly relevant for intuition on what to expect from solutions of these equations.

\begin{lemma}\label{l:M+Properties}
	$M^+$ (and hence $M^-$) obey the following
	\begin{enumerate}
		\item[(i)] If $z\in\R^d$ is fixed, and $Tu:= u(\cdot+z)$, then $M^+Tu(x) = M^+u(x+z)$ (translation invariant).
		\item[(ii)] If $R$ is a rotation or reflection on $\R^d$, then $M^+u(R\cdot)(x)= M^+u(Rx)$ (rotation invariant).
		\item[(iii)] If $r>0$, then $M^+u(r\cdot)(x) = r^\al M^+u(rx)$ (scaling).
	\end{enumerate}
	
\end{lemma}  

\begin{proof}[Proof of Lemma \ref{l:M+Properties}]
	Property (i) follows from a direct equality in $LTu(x)=Lu(x+z)$ whenever $K\in\mathcal L$ (importantly note that $K\in\mathcal L$ requires $K(x,h)=K(h)$).  Property (ii) follows because $\mathcal L$ is closed under composing $K$ with a rotation or reflection.  Property (iii) follows from the observation that if $K\in\mathcal L$, then
	\[
	\widetilde K(h):= r^{-d-\al}K(\frac{h}{r}) \in\mathcal L
	\]
	as well, combined with the fact that for $L$, $\tilde L$ corresponding to $K$, $\tilde K$
	\[
	Lu(r\cdot)(x) = r^\al \tilde Lu(rx).
	\]
It is worth remarking that when $\al=1$, one must be careful with rescaling the integral due to the presence of $\Indicator_{B_1}(h)$.  However, in this case the rescaling still holds because (A4) implies that 
\begin{equation*}
	\int_{B_{1}\setminus B_r} hK(h) \dd h = 0,
\end{equation*}
and this allows to keep the term $\Indicator_{B_1}(h)$ fixed in $\tilde L$ without effecting its value.
\end{proof}

In the rest of this section, we make some elementary estimates that give us some bounds on $Lu(x)$ in terms of bounds for $u$ and its derivatives. These estimates explain the need of the assumptions \eqref{e:assumption-above} and \eqref{e:assumption-odd}.  We start with the following lemma.

\begin{lemma}\label{l:KernelIntegrationBounds}
Let $K$ be a kernel satisfying assumptions (A2) and (A4). Then, the following inequalities hold
\begin{align}
\int_{B_r} |h|^2 K(h) \dd h &\leq C \Lambda r^{2-\alpha}, \label{e:integral-square-Br}\\
\left\vert \int_{B_r} h K(h) \dd h \right\vert  &\leq C \Lambda r^{1-\alpha}\ \text{ if } \alpha < 1,  \label{e:integral-linear-Br} \\
\left\vert \int_{\R^d \setminus B_r} h K(h) \dd h \right \vert  &\leq C \Lambda r^{1-\alpha} \ \text{ if } \alpha > 1, \label{e:integral-linear-tail}\\
\int_{\R^d \setminus B_r} K(h) \dd h &\leq C \Lambda \frac{2-\alpha}{\alpha} r^{-\alpha}.\label{e:integral-value-tail}
\end{align}
In this lemma, the constant $C$ is independent of all the other constants.
\end{lemma}

\begin{proof}
	The four assertions are all proved in a similar fashion, and they follow from a straight-forward decomposition of the integrals in dyadic rings $B_{2^{k+1} r} \setminus B_{2^k r}$ followed by applications of \eqref{e:assumption-above} and \eqref{e:assumption-odd}.  We will only write down explicitly the proof of \eqref{e:integral-linear-tail} as an example. 

Assume $\alpha>1$. We use \eqref{e:assumption-odd} and decompose the integral in dyadic rings $B_{2^{k+1} r} \setminus B_{2^k r}$
\begin{align*}
\left\vert \int_{\R^d \setminus B_r} h K(h) \dd h \right\vert  &\leq \sum_{k=0}^\infty \left\vert \int_{B_{2^{k+1}r} \setminus B_{2^k r}} h K(h) \dd h  \right\vert, \\
&\leq \sum_{k=0}^\infty \Lambda |1-\alpha| (2^k r)^{1-\alpha}, \\
&\leq \Lambda r^{1-\alpha} \frac {|1-\alpha|}{1-2^{1-\alpha}}.
\end{align*}
Since the last factor on the right is bounded uniformly for $\alpha \in (1,2)$, we finished the proof.
\end{proof}

\begin{lemma} \label{l:classical-bounds}
Assume $\alpha \geq \alpha_0$. Let $K$ be any kernel which satisfies \eqref{e:assumption-above} and \eqref{e:assumption-odd}. Let $u$ be a function which is $C^2$ around the point $x$ and $p = \grad u(x)$. Moreover, assume that $u$ satisfies the following bounds globally
\begin{align*}
D^2 u &\leq A I, \\
|u| &\leq B.
\end{align*}
Then,
\[ \int_{\R^d} \delta_h u(x) K(h) \dd h \leq C \left( \frac B A \right)^{-\alpha/2} \left( B + \left( \frac B A \right)^{1/2} |p| \right).\]
Here $C$ is a constant which depends on $\Lambda$ and $\alpha_0$. Moreover, when $\alpha=1$ we can drop the term depending on $p$ and get
\[ \int_{\R^d} \delta_h u(x) K(h) \dd y \leq C (AB)^{1/2}.\]
\end{lemma}

\begin{proof}
Since $\delta_h u(x)$ has a different form depending on $\alpha>1$, $\alpha=1$ and $\alpha<1$, we must divide the proof in these three cases.

We start with the case $\alpha<1$. In this case $\delta_h u(x) = u(x+h) - u(x)$. Let $r>0$ be arbitrary, then
\begin{align} 
	\int_{\R^d} \delta_h u(x) K(h) \dd y &= \int_{B_r} \delta_h u(x) K(h) \dd h + \int_{\R^d \setminus B_r} \delta_h u(x) K(h) \dd h, \nonumber \\
&\leq \int_{B_r} (p \cdot h + A|h|^2) K(h) \dd h + \int_{\R^d \setminus B_r} 2B \ K(h) \dd h,
\intertext{Using \eqref{e:integral-linear-Br}, \eqref{e:integral-square-Br} and \eqref{e:integral-value-tail}, we get}
&\leq C \left( |p| r^{1-\alpha} + A r^{2-\alpha} + B r^{-\alpha} \right). \label{e:ClassicalBoundsGoal1}
\end{align}
We finish the proof in the case $\alpha<1$ by picking $r = (B/A)^{1/2}$.

The case $\alpha>1$ is similar. In this case $\delta_h u(x) = u(x+h) - u(x) - p \cdot h$ and we get
\begin{align*} \int_{\R^d} \delta_h u(x) K(h) \dd h &= \int_{B_r} \delta_h u(x) K(h) \dd h + \int_{\R^d \setminus B_r} \delta_h u(x) K(h) \dd h, \\
&\leq \int_{B_r} A|h|^2 K(h) \dd h + \int_{\R^d \setminus B_r} (p \cdot h + 2B) \ K(h) \dd h,
\end{align*}
This time using \eqref{e:integral-square-Br}, \eqref{e:integral-linear-tail} and \eqref{e:integral-value-tail}, we again arrive at (\ref{e:ClassicalBoundsGoal1}), and conclude by picking the same $r = (B/A)^{1/2}$.

We are left with the case $\alpha=1$. In this case 
$\delta_h u(x) = u(x+h) - u(x) - p \cdot h \ \Indicator_{B_1}(h)$. For arbitrary $r>0$, we have
\[ \int_{\R^d} \delta_y u(x) K(h) \dd h = \int_{B_r} (u(x+h) - u(x) - p \cdot h) K(y) \dd h + \int_{\R^d \setminus B_r} (u(x+h) - u(x)) K(h) \dd h \pm \int_{B_1 \triangle B_r} h \cdot p \ K(h) \dd h \]
The last term on the right is equal to zero because of the assumption \eqref{e:assumption-odd}. Therefore, we can drop this term and use the other two to estimate the integral.

\begin{align*} \int_{\R^d} \delta_h u(x) K(h) \dd h &\leq \int_{B_r} A|h|^2 K(h) \dd h + \int_{\R^d \setminus B_r} 2B \ K(h) \dd h,
\intertext{Using \eqref{e:integral-square-Br} and \eqref{e:integral-value-tail}, we get}
&\leq C \left(A r + B r^{-1} \right).
\end{align*}
Picking $r = (B/A)^{1/2}$, we obtain
\[\int_{\R^d} \delta_h u(x) K(h) \dd h \leq C (AB)^{1/2}.\]
This concludes the proof in all cases.
\end{proof}

\begin{remark}
Lemma \ref{l:classical-bounds} requires an inequality to hold for $D^2 u$ in the full space $\R^d$. This does not require the function $u$ to be $C^2$ globally. What it means is that $u(x) - \frac A2 |x|^2$ is concave.
\end{remark}

\begin{cor} \label{c:classical-estimates-pucci}
Let $M^+_{\mathcal L}$ and $M^-_{\mathcal L}$ be the extremal operators defined above. Let $p = \grad u(x)$ and assume that $u$ satisfies the global bounds
\begin{align*}
-A_- I \leq D^2 u &\leq A_+ I, \\
|u| &\leq B.
\end{align*}
Then
\begin{align*}
M^+_{\mathcal L} u(x) &\leq C \left( \frac B A_+ \right)^{-\alpha/2} \left( B + \left( \frac B A_+ \right)^{1/2} |p| \right), \\
M^-_{\mathcal L} u(x) &\geq -C \left( \frac B A_- \right)^{-\alpha/2} \left( B + \left( \frac B A_- \right)^{1/2} |p| \right).
\end{align*}
Moreover, if $\alpha=1$, the estimate can be reduced to
\begin{align*}
M^+_{\mathcal L} u(x) &\leq C \left( B A_+ \right)^{1/2}, \\
M^-_{\mathcal L} u(x) &\geq -C \left( B A_- \right)^{1/2}.
\end{align*}
\end{cor}

\begin{proof}
The estimate for $M^+_{\mathcal L}$ follows taking the supremum in $K$ in Lemma \ref{l:classical-bounds}. The estimate for $M^-_{\mathcal L}$ follows then since
\[ M^-_{\mathcal L} u(x) = -M^+_{\mathcal L}[-u](x).\]
\end{proof}

\section{Viscosity solutions}

We use a standard definition of viscosity solutions for integral equations which is the parabolic version of the one in \cite{caffarelli2009regularity} and equivalent under most conditions to the parabolic version of \cite{BaIm-07}.

\begin{defn}[cf. \cite{caffarelli2011regularity}-- Definition 21 and (1.2).] \label{d:ellipticity} We say $I$ is a nonlocal operator which is elliptic with respect to the class of operators in this article if $Iu(x)$ is well defined for any function $u \in Growth(\alpha)$ such that $u \in C^2(x)$ and moreover,
\[ M^-(u_1-u_2)(x) - C|\grad (u_1-u_2)(x)| \leq Iu_1(x) - Iu_2(x) \leq M^+(u_1-u_2)(x) + C|\grad (u_1-u_2)(x)|.\]
The constant $C$ must be equal to zero if $\alpha \leq 1$.

We say that $I$ is translation invariant if $I[u(\cdot-x_0)] = Iu(\cdot-x_0)$.
\end{defn}

Note that the operators $M^+$ and $M^-$ in particular are nonlocal operators, uniformly elliptic with respect to this class. These are the only operators that are needed for the main result in this article (Theorem \ref{thm:intro}). The main result has implications to nonlinear equations in terms of operators as in Definition \ref{d:ellipticity} which are given in section \ref{s:nonlinear}.

\begin{defn}[cf. \cite{caffarelli2009regularity}-- Definition 2.2 and \cite{caffarelli2011regularity}-- Definition 25 ] \label{d:viscosity}
Let $I$ be a nonlocal operator as in Definition \ref{d:ellipticity}. Assume that $u\in Growth(\al)$. We say $u: \R^d \times [T_1,T_2]$ satisfies the following inequality in the viscosity sense, and also refer to it as a viscosity supersolution of
\[ u_t - Iu \geq 0 \qquad \text{ in } \Omega \subset \R^d \times \R,\]
if every time there exist a $C^{1,1}$ function $\varphi : D \subset \Omega \to \R$ so that $\varphi(x_0,t_0) = u(x_0,t_0)$ and also $u \geq \varphi$ in $D \cap \{t \leq t_0\}$, then the auxiliary function
\[ v(x) = \begin{cases}
\varphi(x,t_0) & \text{if } (x,t_0) \in D,\\
u(x,t_0)  & \text{if } (x,t_0) \notin D.
\end{cases}
\]
satisfies
\[ v_t(x_0,t_0) - Iv(x_0,t_0) \geq 0.\]
\end{defn}

One of the most characteristic properties of viscosity solutions is that they obey the comparison principle. In the context of this article, we state it as follows.

\begin{prop} \label{p:comparisonpple}
Let $I$ be a translation invariant nonlocal operator which is uniformly elliptic in the sense of Definition \ref{d:ellipticity}. Let $u,v \in \R^n \times [0,T]$ be two continuous functions such that
\begin{itemize}
\item For all $x\in \R^n$, $u(x,0) \geq v(x,0)$. 
\item For all $x \in \R^n \setminus B_1$ and $t  \in [0,T]$, $u(x,t) \geq v(x,t)$.
\item $u_t - Iu \geq 0$ and $v_t - Iv \leq 0$ in $B_1 \times [0,T]$.
\end{itemize}
The $u(x,t) \geq v(x,t)$ for all $x \in B_1$ and $t \in [0,T]$.
\end{prop}

The proof of Proposition \ref{p:comparisonpple} is by now standard. We refer the reader to \cite{lara2014regularity} (Corollary 3.1), \cite{silvestre2011differentiability} (Lemmas 3.2, 3.3), \cite{caffarelli2009regularity} (Theorem 5.2) and \cite{BaIm-07} for the main ideas. For the purposes of this article, we do not use the full power of Proposition \ref{p:comparisonpple}. We only use the comparison principle to compare a supersolution $u$ with a special barrier function constructed in section \ref{sec:Barrier}. This barrier function is explicit and is smooth, except on a sphere where it has an \emph{angle} singularity. The comparison principle follows easily from Definition \ref{d:viscosity} when $v$ is this special barrier function or any smooth subsolution of the equation.

In \cite{caffarelli2009regularity}, and many subsequent works, it was frequently used that wherever a viscosity solution $u$ can be touched with a $C^2$ test function from one side, the equation can be evaluated classically with the original $u$ at that particular point (a notable departure from the second order theory!).   This fact plays a role in some measure estimates used to prove the regularity results in those works. With our current setting, it is not possible to evaluate the equation pointwise in $u$ because of the gradient terms, however many possible useful variations on that theme can be shown-- similar to \cite[Appendix 7.2]{rang2013h}.  In this case, the following lemma is what we will use to obtain pointwise evaluation of the regularized supersolution.

\begin{lemma} \label{l:viscosity-classical}
Assume $u$ satisfies the following inequality in the viscosity sense
\[ u_t + C_0|\grad u| - M^- u \geq -C \text{ in } \Omega.\]
Assume also that there is a test function $\varphi:\R^d \times [t_1,t_2] \to \R$ so that $\varphi(x_0,t_0)=u(x_0,t_0)$ and $\varphi(x,t) \leq u(x,t)$ for all $t \in (t_0-\eps,t_0]$.

Then, the following inequality holds
\begin{align*}
\varphi_t(x_0,t_0) + &C_0 |\grad \varphi(x_0,t_0)| - M^- \varphi(x_0,t_0) \\
&- \inf\left\{ \int_{\R^d} \left( u(x+y,t_0) - \varphi(x+y,t_0) \right) K(y) \dd y\ :\ K\in \mathcal K \right\} \geq -C.
\end{align*}
\end{lemma}

\begin{proof}
We can use $\varphi$ as the test function for Definition \ref{d:viscosity} in any small domain $D = B_r(x_0) \times (t_0-\eps,t_0]$. Constructing the auxiliary function $v$ we observe that
\begin{align*}
v_t(x_0,t_0) &= \varphi_t(x_0,t_0), \\
\grad v(x_0,t_0) &= \grad \varphi(x_0,t_0), \\
M^- v(x_0,t_0) &= \inf \left\{ \int_{\R^d} \delta_y \varphi(x) K(y) \dd y + \int_{\R^d \setminus B_r} \left( u(x+y) - \varphi(x+y) \right) K(y) \dd y\ :\ K\in \mathcal K \right\} \\
&\geq  M^- \varphi(x_0,t_0) +\inf \left\{\int_{\R^d \setminus B_r} \left( u(x+y) - \varphi(x+y) \right) K(y) \dd y\ :\ K\in\mathcal K \right\}
\end{align*}
Observe that the last term is monotone increasing as $r \to 0$.

From Definition \ref{d:viscosity}, we have that for any $r>0$, $v_t(x_0,t_0) + C_0|\grad v(x_0,t_0)| - M^- v(x_0,t_0) \geq -C_1$. The result of the Lemma follows by taking $r \to 0$.
\end{proof}

\section{Relating a point-wise value with an estimate in measure: the growth lemma}

In order to obtain the H\"older continuity of $u$, we need to show the following point-to-measure lemma which seems to originate in the work of Landis \cite{Landis-1971SecondOrderBookTranslation} (in some circles, it is known as a the \emph{growth lemma}).  It is a cornerstone of the regularity theory, it leads to the weak Harnack inequality, and it is one of the few places where the equation plays a fundamental role.
\begin{lemma} \label{l:growth-lemma}
There exists positive constants $A_0$ and $\delta_0$ depending on $\lambda$, $\Lambda$, $d$, $\alpha_0$ and $C_0$ so that if $\alpha > \alpha_0$ and if $u :\R^d \times (-1,0] \to \R$ is a function such that
\begin{enumerate}
\item $u \geq 0$ in the whole space $\R^d \times (-1,0]$.
\item  $u$ is a supersolution in $Q_1$, i.e.
\begin{equation}\label{e:GrowthLemmaSuperSolU}
u_t + C_0 |\grad u| - M^- u(x) \geq 0, \text{ in } Q_1.
\end{equation}
\item $\min_{Q_{1/4}} u \leq 1$,
\end{enumerate}
then 
\[ | \{u \leq A_0\} \cap Q_1 | \geq \delta_0. \]
\end{lemma}

The following function, $q$, plays an important role in the proof of Lemma \ref{l:growth-lemma}.  It is actually an inf-convolution of $u$ with a quadratic, and it is defined as
\begin{equation} \label{e:q}  
	q(x,t) = \min_{y \in \overline B_1} u(y,t) + 64|x-y|^2. 
\end{equation}
Note $q$ is a nonnegative function.  We will prove a collection of properties of the function $q$, which will lead us to the proof of Lemma \ref{l:growth-lemma}.

The next barrier is used to find a bound for the rate at which $q$ can decrease with respect to $t$.
\begin{lemma} \label{l:truncated-parabola}
For a universal constant $C_1$, the function
\[ \varphi(x,t) = \max(0, f(t) - 64 |x|^2),\]
is a subsolution to
\[ \varphi_t + C_0 |\grad \varphi| - M^- \varphi \leq 0 \text{ in } \R^n \times (-\infty,0]. \]
The inequality holds classically at all points where $\varphi > 0$.

Here $f(t)$ is the (unique) positive solution to the (backward) ODE
\begin{align}\label{e:special-ode}
\begin{cases}
f(0) &= 0, \\
f'(t) &= -C_1 \left(f(t)^{1/2} + f(t)^{1-\alpha/2} \right)
\end{cases}
\end{align}
where $C_1$ is a constant depending on $\Lambda$ and $\alpha_0$ (such that $\alpha \geq \alpha_0$).
\end{lemma}

\begin{proof}
Note that for every fixed value of $t \in (-\infty,0]$, it holds that 
\[
\|\varphi\|_{L^\infty} = f(t),\ \ \|\grad \varphi\|_{L^\infty} \leq C \sqrt{f(t)},\ \text{and}\ \ 0\geq D^2 \varphi \geq -128 I.
\]
Applying Corollary \ref{c:classical-estimates-pucci},
\[ M^- \varphi \geq -C f(t)^{1-\alpha/2}.\]
Then, at all points where $\varphi > 0$ we have
\[ \varphi_t + C_0 |\grad \varphi| - M^- \varphi \leq f'(t) + C_0 C f(t)^{1/2} + C f(t)^{1-\alpha/2}.\]
The lemma then follows by choosing $C_1$ so that $f'$ dominates the right hand side.
\end{proof}

It is worth commenting that the ODE for $f$ in Lemma \ref{l:truncated-parabola} has a unique solution which is strictly positive for $t<0$. This function $f$ is differentiable and locally Lipschitz. The universal constant $C_2$ of the following result is the Lipschitz constant of $f$ in the interval $[-T,0]$, where $f(T) = -4$.

\begin{cor} \label{c:half-Lipschitz}
Assume $x \in B_{1/8}$ and $q(x,t) < 3$. Then there are positive universal constants $\tau$ and $C_2$ such that for $s \in (t-\tau,t)$, $q(x,s) - q(x,t) < C_2(t-s)$.
\end{cor}

\begin{proof}
	We let $x$, $t$, and $s$ be fixed as stated. Let $y$ be the point where the minimum for $q(x,t)$ is achieved in \eqref{e:q}.  Using the definition of $q$ we note for all values of $z \in B_1$, $u(z,s) \geq q(x,s) - 64 |x-z|^2$.

The point of the proof is to use the fact that $u$ and $\varphi$ are respectively super and sub solutions of the equation (\ref{e:GrowthLemmaSuperSolU}) on the time interval $(s,0]$.  In order to invoke a comparison result between them, we will make various choices involving $\tau$ and $f$ to enforce $\varphi$ to be below $u$ at the initial time, $s$, and on the boundary, which is $\R^d\setminus B_1$. 	

We define the function
\begin{equation*}
	\bar \varphi(\bar x, \bar t) := \varphi( \bar x - x, \bar t - s + t_0),
\end{equation*}
where $t_0$ is a fixed time, yet to be chosen.  We fix the constant $\tau$ so that
\begin{equation*}
	\tau < f^{-1}(3) - f^{-1}(4),
\end{equation*}
and we fix the time $t_0 < 0$ so that
\begin{equation*}
	f(t_0) = \min(q(x,s),4).
\end{equation*}
Checking the boundary condition for $\bar x\not\in B_1$ and $\bar t > s$ we see that $\abs{x-\bar x}\geq 7/8$ (as $x\in B_{1/8}$), and hence since $f(t_0)\leq 4 \leq 49$ we have (note $f$ is decreasing)
\begin{align*}
	\bar\varphi(\bar x,\bar t) = \varphi(\bar x - x, \bar t - s + t_0) &= \max(0, f(\bar t - s +t_0) - 64\abs{x-\bar x}^2) \leq \max(0, f(t_0) - 49)\leq 0.
\end{align*}
Checking the initial condition at $\bar t=s$ we have (by the definition of $t_0$)
\begin{equation*}
	\bar\varphi(\bar x,s) = \varphi(\bar x - x, t_0) = \max(0, f(t_0) - 64\abs{x-\bar x}^2)\leq \max(0, q(x,s) - 64\abs{x-\bar x}^2) \leq u(\bar x,s),
\end{equation*}
from the definition of $q$.

Comparison therefore tells us that $u\geq \bar\varphi$ on $B_1\times(s,0)$, and in particular for $\bar x=y$ and $\bar t = t$,
\[ u(y,t) \geq \varphi(x-y,t-s+t_0) \geq f(t-s+t_0)-64 |x-y|^2. \]
Hence
\begin{equation*}
q(x,t) = u(y,t) + 64 |x-y|^2 \geq f(t-s+t_0),
\end{equation*}
and we will use
\begin{equation*}
	q(x,t) \geq f(t-s+t_0) \geq f(t_0) - \abs{f'(t_0)}(t-s).
\end{equation*}
In the case that $f(t_0)=q(x,s)$, we can conclude the corollary with $C_2:=\max\{f'(t): t \in (-f^{-1}(4),0)\}$.  However, $\tau$ was chosen specifically so that it is impossible for $f(t_0) < q(x,s)$.  Indeed we see that if it occurred that $f(t_0)=4$ then because $f$ is decreasing and $t-s\leq \tau$, it holds
\begin{align*}
	3 > q(x,t) \geq f(t-s+t_0) \geq f(t_0) + f(\tau+t_0) - f(t_0) \geq 4 + f(f^{-1}(3)) - 4 = 3,
\end{align*}
which is a contradiction.  Thus $f(t_0)=q(x,s)$ is the only possibility, and we conclude.

\end{proof}

Corollary \ref{c:half-Lipschitz} should be interpreted as that $q_t \geq -C_2$ everywhere. The next lemma gives us a bound above for $q_t$ in a set of positive measure.

\begin{lemma} \label{l:qt-ae}
Under the assumptions of Lemma \ref{l:growth-lemma}, (but assuming here $u(0,0)=1$) the function $q$ from \eqref{e:q} satisfies that $|\{ q_t \leq A_1\} \cap Q_1| \geq \del_1 > 0$, where $A_1$ and $\del_1$ are universal constants.
\end{lemma}

\begin{proof}
Since $u(0,0)=1$, for any $x \in B_{1/4}$ we have $q(x,0) \leq 1+ 64|x|^2 < 5$. Moreover, the minimum is achieved at some $y \in B_{1/2}$ since $1+64|y-x|^2 > 5$ if $|y|>1/2$.  By a similar reasoning we also have that for every $x \in B_{1/8}$, it holds that $q(x,0) < 2$.  Corollary \ref{c:half-Lipschitz} implies that for $t\in(-\tau,0]$
\[
q(x,t) \leq q(x,0) + C_2\abs{t}< 2 + C_2\abs{t}.
\]
Thus if we restrict $t\in (-\tau',0]$, where $\tau'=1/C_2$, then we have that $q(x,t)<3$ and a second application of Corollary \ref{c:half-Lipschitz} shows that $q(x,t) + C_2t$ is monotone increasing.  Thus $q_t(x,t)$ exists pointwise for a.e. $t\in(-\tau',0]$ and $q_t$ exists as a signed measure.  Furthermore,  
\[
q_t(x,t) \geq -C_2\ \ \text{for a.e.}\ t\in(-\tau',0].
\]
Integrating the measure $q_t(x,t)$ and ignoring its singular part shows (note, $q\geq0$ always)
\begin{align*}
C=2\abs{B_{1/8}} &\geq \int_{B_{1/8}} q(x,0) - q(x,-\tau') \dd x, \\
&\geq \int_{-\tau'}^0 \int_{B_{1/8}}  q_t(x,s) \dd x \dd s, \\
&\geq A_1 |\big( (-\tau',0] \times B_{1/8} \big) \cap \{ q_t > A_1\}| - C_2 |\big((-\tau',0] \times B_{1/8}\big) \setminus \{ q_t > A_1\}| \\
&= -C_2 \tau' |B_{1/8}| + (A_1+C_2) |\big((-\tau',0] \times B_{1/8}\big) \cap \{ q_t > A_1\}|.
\end{align*}

Therefore, rearranging shows that
\[ |\big((-\tau',0] \times B_{1/8}\big) \cap \{q_t > A_1\}| \leq \frac{C+C_2\tau'}{A_1+C_2}.\]
We can make the right-hand side arbitrarily small by choosing $A_1$ large. In particular, we choose $A_1$ sufficiently large (depending only on universal constants) so that we have
\[ |\big((-\tau',0] \times B_{1/8}\big) \cap \{q_t \leq A_1\}| \geq \frac 12 \tau' |B_{1/8}| =: \del_1.\]
\end{proof}

After Corollary \ref{c:half-Lipschitz} and Lemma \ref{l:qt-ae}, we obtain a set of positive measure where $|q_t|$ is bounded. At this points, we can use ideas from the stationary case to proceed with the rest of the proof.

The next lemma replaces Lemma 8.1 in \cite{caffarelli2009regularity}.  We in fact prove a slightly modified version of the lemma which enforces a quadratic growth of $\del_h u$ simultaneously on \emph{two} rings. In the proof of Theorem 8.7 and Lemma 10.1 in \cite{caffarelli2009regularity}, there is a cube decomposition plus a covering argument. It could be replaced by a double covering argument. In this paper we will have a simpler covering argument using Vitali's lemma only once. This is possible thanks to the stronger measure estimate in the next lemma (in two simultaneous rings).

\begin{lemma} \label{l:local-measure-estimate}
Let $\mu$ be the constant in (\ref{e:assumption-below}) and $c_0<1$ be an arbitrary constant. Let $y$ be the point in $B_{1/2}$ where the minimum of \eqref{e:q} is achieved, and $u$ satisfies \eqref{e:GrowthLemmaSuperSolU}. Assume that $x \in B_{1/4}$, $q(x,t) < 3$ and $q_t(x,t) < A_1$.  Then, for $A_2$ sufficiently large (depending on $C_1$, $\mu_1$, $\lambda$, $\Lambda$, $c_0$ and $\al_0$ but not on $\alpha$) we have that there exists some $r \leq r_0$ so that both
\begin{equation} 
|\{ h \in B_{2r} \setminus B_r : \delta_h u(y,t) \leq A_2 r^2 \text{ and } \delta_{-h} u(y,t) \leq A_2 r^2 \} | \geq \frac \mu 2 |B_{2r} \setminus B_r|
\end{equation}
and 
\begin{equation} 
|\{ h \in B_{2c_0r} \setminus B_{c_0r} : \delta_h u(y,t) \leq A_2 (c_0r)^2 \text{ and } \delta_{-h} u(y,t) \leq A_2 (c_0r)^2 \} | \geq \frac \mu 2 |B_{2c_0r} \setminus B_{c_0r}|,
\end{equation}
hold simultaneously for $r$ and $c_0r$.
Here $r_0 = 4^{-1/(2-\alpha)}$, and we note that $r_0 \to 0$ as $\alpha \to 2$.
\end{lemma}

In Lemma \ref{l:local-measure-estimate}, we abuse notation by writing 
\[ \delta_h u(y,t) = u(y+h,t) - u(y,t) - 128(x-y) \cdot h,\]
even though $\grad u(y,t)$ may not exist. Note that if $u$ happens to be differentiable at $(y,t)$, then $\grad u(y,t) = 128 (x-y)$ because of \eqref{e:q}.
The value of $c_0$ will be selected as a universal constant in Lemma \ref{l:mBr}.

\begin{proof}
From the construction of $x$ and $y$, we have that 
$u(y,t) = q(x,t) - 64 |x-y|^2$. Moreover, $u(z,s) \geq q(x,s) - 64 |x-z|^2$ for any $z \in\R^n$ and $s \leq t$. Since we are assuming that $q_t(x,t) < A_1$ (in particular that $q_t$ exists at that point), there is an $\eps>0$ so that $q(x,s) > q(x,t) - A_1 (t-s)$ for $s \in (t-\eps,t]$. Consequently, $u(z,s) \geq q(x,t) - 64 |x-z|^2 - A_1 (t-s)$ for $s \in (t-\eps,t]$.

Let 
\[\varphi(z,s) := \max \left( q(x,t) - 64 |x-z|^2 - A_1 (t-s) , -256\right).\]
The choice of the number $-256$ is made so that the maximum is always achieved by the paraboloid every time $z \in B_1$. From the analysis above, we have that $u \geq \varphi$ in $\R^n \times (t-\eps,t]$ and $u(y,t) = \varphi(y,t)$. Note that since $q(x,t) < 3$, then $|\grad \varphi(y,t)| \leq 16\sqrt{3}$. Also, from Lemma \ref{l:classical-bounds}, since $D^2 \varphi \geq -128 I$, then $M^- \varphi(y,t) \geq -C$ for some universal constant $C$. We apply Lemma \ref{l:viscosity-classical} and we get
\begin{align*}
0 &\leq \varphi_t(y,t) + C_0 |\nabla \varphi(y,t)| - M^- \varphi(y,t) - \inf \left\{ \int_{\R^d} (u(y+h,t)-\varphi(y+h,t)) K(h) \dd h : K\in\mathcal K \right\}, \\
&\leq A_1 + C_0 |\nabla \varphi(y,t)| - M^- \varphi(y,t) - \inf \left\{ \int_{\R^d} (u(y+h,t)-\varphi(y+h,t)) K(h) \dd h : K\in\mathcal K \right\}, \\
&\leq C - \inf \left\{ \int_{\R^d} (u(y+h,t)-\varphi(y+h,t)) K(h) \dd h : K\in\mathcal K \right\}.
\end{align*}

Note that $u(y+h,t)-\varphi(y+h,t) \geq 0$ for all values of $h \in \R^n$. We abuse notation by calling
\[ \delta_h u(y,t) = u(y+h,t) - u(y,t) - h \cdot \grad \varphi(y,t).\]
Note that
\[ u(y+h,t)-\varphi(y+h,t) = \delta_h u(y,t) - \delta_h \varphi(y,t).\]
And $\delta_h \varphi(y,t) = -64 |h|^2$ whenever $y+h \in B_1$. 

Using that the integrand is positive, we can reduce its domain of integration to an arbitrary subset of $\R^n$.
\begin{align*} C &\geq \inf \left\{ \int_{B_{r_0}} (u(y+h,t)-\varphi(y+h,t)) K(h) \dd h : K\in\mathcal K \right\}, \\
&= \inf \left\{ \int_{B_{r_0}} (\delta_h u(y,t)+64 |h|^2) K(h) \dd h : K\in\mathcal K \right\}.
\end{align*}
Let us call $w(h) := \delta_h u(x,t) + 64 |h|^2 \geq 0$ for $h \in B_{r_0}$. We have that there exists an admissible kernel $K$ such that

\begin{equation} \label{e:r1}  C \geq \int_{B_{r_0}} w(h) K(h) \dd h.
\end{equation}

Let $r \leq r_0 = 4^{-1/(2-\alpha)}$. From \eqref{e:assumption-below}, we know that
\begin{equation}\label{e:rRecallKernelBounds} 
|\{ h \in B_{2r} \setminus B_r : K(h) > (2-\alpha) \lambda r^{-d-\alpha} \text{ and }  K(-h) > (2-\alpha) \lambda r^{-d-\alpha} \} | > \mu |B_{2r} \setminus B_r|.
\end{equation}
In order to obtain a contradiction, let us assume that the result of the Lemma is false. That is, for all $r \leq r_0$, either 
\begin{equation} \label{e:r2}  | \{h \in B_{2r} \setminus B_r : w(h) > (A + 64)r^2 \text{ or } w(-h) > (A+64) r^2 \}| > (1-\mu/2) |B_{2r} \setminus B_r|
\end{equation}
or
\begin{equation} \label{e:r2c0}  | \{h \in B_{2c_0r} \setminus B_{c_0r} : w(h) > (A+64) (c_0r)^2 \text{ or } w(-h) > (A+64) (c_0r)^2 \}| > (1-\mu/2) |B_{2c_0r} \setminus B_{c_0r}|.
\end{equation}

Therefore, the intersection of the set in (\ref{e:rRecallKernelBounds})-- with $r$ appropriately chosen in each case-- with either of that in (\ref{e:r2}) or (\ref{e:r2c0}) must have measure at least $\mu/2 |B_{2r} \setminus B_r|$ or $\mu/2 |B_{2c_0r} \setminus B_{c_0r}|$, depending on which of the two possibilities occurred.  Let us set $\tilde r$ to be either $r$ or $c_0r$, depending upon whether we will invoke (\ref{e:r2}) or (\ref{e:r2c0}). Let us call $G_{\tilde r}$ this intersection between the sets (\ref{e:rRecallKernelBounds}) and either (\ref{e:r2}) or (\ref{e:r2c0}). Note that $G_{\tilde r} \subset B_{2\tilde r} \setminus B_{\tilde r}$ and $G_{\tilde r}$ is symmetric (i.e. $G_{\tilde r} = -G_{\tilde r}$). Moreover, for all $h \in G_{\tilde r}$ either $w(h) > (A+64) \tilde r^2$ and $K(h) > (2-\alpha) \lambda \tilde r^{-d-\alpha}$ or $w(-h) > (A+64) \tilde r^2$ and $K(-h) > (2-\alpha) \lambda \tilde r^{-d-\alpha}$. Therefore

\begin{align*}
\int_{B_{2\tilde r} \setminus B_{\tilde r} } w(h) K(h) \dd h &\geq \int_{G_{\tilde r} } w(h) K(h)\dd h,\\
&= \frac 12 \int_{G_{\tilde r}} w(h) K(h) + w(-h) K(-h)\dd h,\\
&\geq \frac 12 \int_{G_{\tilde r}} A \lambda (2-\alpha) \tilde r^{-d+2-\alpha} \dd h \\
&\geq A \lambda (2-\alpha) \tilde r^{2-\alpha} \mu \omega_d,
\end{align*}
where $\omega_d$ is a constant depending on dimension only.

We invoke the contradiction assumption for each of the radii $r_j = 2^{-j-1} r_0$ with $j=0,1,2,\dots$.  For each $r_j$, we get the estimates corresponding to $\tilde r_j$, which is either $r_j$ or $c_0 r_j$, depending on the case of the contradiction assumption.  Partitioning $B_{r_0}$ we get
\allowdisplaybreaks
\begin{align*}
\int_{B_{r_0}} w(h) K(h) \dd h &= \sum_{j=0}^\infty \int_{B_{2r_j} \setminus B_{r_j}} w(h) K(h) \dd h, \\
&\geq \frac 12 \sum_{j=0}^\infty \int_{B_{2\tilde r_j} \setminus B_{\tilde r_j}} w(h) K(h) \dd h, \\
&\geq (A+64) \lambda (2-\alpha) \mu \omega_d \sum_{j=0}^\infty \left( \tilde r_j \right)^{2-\alpha}, \\
&\geq (A+64) \lambda (2-\alpha) \mu \omega_d \sum_{j=0}^\infty \left( c_0 2^{-j-1} r_0 \right)^{2-\alpha}, \\
&= C(d) c_0^{2-\al}(A+64) \mu \lambda \frac{2-\alpha}{1-2^{\alpha-2}}.
\end{align*}

We get a contradiction with \eqref{e:r1} if $A$ is large enough. Note that the last factor is bounded away from zero, independently of $\alpha$ as long as $\alpha \in (0,2)$.  Thus the value of $A=A_2$ is independent of $\alpha$, and it is chosen to obtain this contradiction.  This concludes the lemma. 
\end{proof}

\begin{lemma} \label{l:mBr}
Under the same conditions as in Lemma \ref{l:local-measure-estimate}, $|m(B_{c_0 r}(y))| \leq C_3 r^d$. Here $r$ is the same value as in Lemma \ref{l:local-measure-estimate}, $c_0$ is fixed from Lemma \ref{l:SomethingLikeGradientMapRegularity} depends only on other universal constants, and $C_3$ depends on $c_0$, $C_4$ (of Lemma \ref{l:SomethingLikeGradientMapRegularity}) and the constant $A_1$ of Lemma \ref{l:local-measure-estimate}.
\end{lemma}

In order to prove Lemma \eqref{l:mBr}, we only use the equation through Lemma \ref{l:local-measure-estimate}. Indeed, after fixing a time $t$ and rescaling, it reduces to the following geometric statement about functions.

\begin{lemma}\label{l:SomethingLikeGradientMapRegularity}
Let $u :\R^d \to \R$ be a continuous bounded function such that $\grad u(0)$ exists. Let $q(x) = \min_{y\in \bar B_1} u(y) + 64|x-y|^2$. Assume the following conditions hold true.
\begin{itemize}
\item There is at least one point $x_0 \in \R^d$ for which $q(x_0) = u(0) + 64|x_0|^2 = \min_{y\in \bar B_1} \{ u(y) + 64\abs{x_0-y}^2\}$.
\item If we consider the (symmetric) set
\[ G := \set{ h \in B_{2} \setminus B_1 : \delta_h u(0) \leq A \text{ and } \delta_{-h} u(0) \leq A },\]
then $|G| \geq \frac \mu 2 |B_{2} \setminus B_1|$. (Here, as in Lemma \ref{l:local-measure-estimate}, $\delta_h u(y,t) = u(y+h,t) - u(y,t) - 128(x-y) \cdot h$ )
\end{itemize}

Then there are constants $c_0$ and $C_4$ depending on $A$ and $\mu$ and $d$ so that if for some pair of point $x_1$, $y_1$ we have
\[ q(x_1) = u(y_1) + 64 |x_1-y_1|^2,\]
then $|y_1|<c_0$ implies $|x_1-x_0| < C_4$.
\end{lemma}

\begin{proof}
Assume $|y_1|<c_0$. Let $p_1$ and $p_2$ be the following quadratic polynomials.
\begin{align*}
p_0(z) = q(x_0) - 64|x_0-z|^2, \\
p_1(z) = q(x_1) - 64|x_1-z|^2.
\end{align*}
From the definition of $q$, we have that $p_0(z) \leq u(z)$ and $p_1(z) \leq u(z)$ for all $z \in \R^d$. Moreover, $p_0(y_0) = u(y_0)$ and $p_1(y_1) = u(y_1)$.

Observe that $p_1 - p_0$ is the affine function
\[
p_1(z) - p_0(z) = q(x_1) - q(x_0) + 64(|x_0|^2 - |x_1|^2) + 128(x_1-x_0) \cdot z.
\]
Since $p_1(y_1) = u(y_1) \geq p_0(y_1)$, then
\[
p_1(y_1+z) - p_0(y_1+z) \geq  128(x_1-x_0) \cdot z.
\]
Using that $u(y_1+z) \geq p_1(y_1+z) \geq p_0(y_1+z) + 128(x_1-x_0) \cdot z$, we get that
\[ \begin{aligned}
\delta_{(y_1+z)} u(0) &\geq \delta_{(y_1+z)} p_0(0) + 128(x_1-x_0) \cdot z \\ &\geq -64 + 128(x_1-x_0) \cdot z \qquad \text{ for } z \in B_1.
\end{aligned}
\]

Let us consider the following set, which is the intersection of a cone (whose vertex is at $y_1$, recall $|y_1|<c_0$) and the ring $B_2 \setminus B_1$.
\[ H = \set{ h \in B_2  \setminus B_1 : h = y_1+z \text{ with } z \cdot (x_1-x_0) > c_0 |z| |x_1-x_0| }.\]
Observe that as $c_0 \to 0$, then $H$ approximates the intersection of the ring $B_2 \setminus B_1$ with the half space $\{z:z \cdot (x_1-x_0) >0\}$. More precisely
\[ |B_2 \setminus B_1 \setminus H \setminus -H| \leq C c_0,\]
for some constant $C$ depending on dimension only.

Let us choose $c_0$ so that $C c_0 < \frac \mu 2 |B_2 \setminus B_1|$. Then $H \cap G$ must have a positive measure (also $G \cap -H$, recall that $G$ is symmetric), and so there exists some $h \in H \cap G$. Then
\begin{align*}
A \geq \delta_h u(0) &\geq -64 + 128(x_1-x_0) \cdot z \\
&> -64 + 128 c_0 |x_1-x_0| |z| \\
&\geq  -64 + 64 c_0 |x_1-x_0|. 
\end{align*}
Therefore $|x_1-x_0| < (A/64+1)/c_0 = : C_4$.
\end{proof}

We simply sketch the main idea to show how Lemma \ref{l:mBr} follows from Lemma \ref{l:SomethingLikeGradientMapRegularity}.
\begin{proof}[Sketch of the proof of Lemma \ref{l:mBr}]
	Assume that $u$ and $q$ are as given in the statements of Lemmas \ref{l:local-measure-estimate} and \ref{l:mBr}.  After a translation, we can assume that $y=0$.  We would then define the rescaled functions
	\[
	\hat u(z) = r^{-2} u(rz)\ \text{and}\ \hat q(z) = r^{-2} q(rz)\ \text{for}\ z\in B_2.
	\]
We note the definition of $\hat q$ will be through a minimum over $B_{1/r}$, but in fact restricting the minimum to $B_1$ changes nothing since $y=0$ is such a point which gives the minimum for $\hat x = x/r$.  Then Lemma \ref{l:SomethingLikeGradientMapRegularity} is applicable with the functions $\hat u$ and $\hat q$, with the point $x_0=\hat x = x/r$, and the set $\hat G = r^{-1} G$ with $G$ being the set arising from the outcome of Lemma \ref{l:local-measure-estimate}.
\end{proof}

In an ABP-based proof, this lemma corresponds to estimating the image of the gradient map of the convex envelop of $u$ in $B_{r}$. This would be the purpose of Lemma 8.4 in \cite{caffarelli2009regularity} or Lemma 3.6 in \cite{BjCaFi-2012NonlocalGradDepOps}. In those cases we would need to adjust $u$ by a supporting hyperplane and argue using a convex envelop.  In our approach, we work without invoking a convex envelop. 

Note that after Corollary \ref{c:half-Lipschitz} and Lemma \ref{l:qt-ae}, where we obtain that $|q_t|$ is bounded in a set of positive measure, the rest of the proof of Lemma \ref{l:growth-lemma} should be interpreted as a nonlocal version of the method in \cite{Savin-2007SmallPerturbationCPDE}. It is more flexible, and arguably more natural, than an ABP-based proof.

We are now in a position to prove Lemma \ref{l:growth-lemma}. 
\begin{proof}[Proof of Lemma \ref{l:growth-lemma}]
We assume $u(0,0)=1$. The result follows for the assumption $\min_{Q_{1/2}} u = 1$ by a simple translation argument.

Let $G$ be the set of points $(x,t) \in B_{1/8} \times (-\tau,0]$ so that $q_t \leq A_1$. From Lemma \ref{l:qt-ae}, we have a universal lower bound on its measure: $|G| > \del_1$.  For each point $(x,t) \in G$, there is at least one point $y \in B_1$ which realizes the minimum value for $q(x,t)$ in \eqref{e:q}.  For each fixed value of $t$, we define the map $m: y \mapsto x$. This is a well defined as function if $u \in C^1$. In general the function nature of $m$ is not necessary, and we should think of $m$ as a set mapping which sends values of $y$ into a set of possible values of $x$ (like the sub-differential of a convex function).

We note that if $y\in m^{-1}(G)$, we have $q_t(x,t) \leq A_1$ for some $x\in G$, and we can apply Lemma \ref{l:local-measure-estimate}, which was presented above.  This gives a ball around $y$ and a collection of points where $u$ does not grow too much, for example we can control the set 

\begin{equation}\label{e:Lem4.1LocalGoodSetAtyDef}
	E_y := \{z \in B_{c_0r}(y) : u(z,t) < A_2 + 43\}.
\end{equation}
This is possible by starting with the ring from Lemma \ref{l:local-measure-estimate} and then noting that $r\leq 1$, $u(y)<3$ (since $q(x)\leq 3$, see first line of the proof of Lemma \ref{l:local-measure-estimate}), $\del_{\pm h} u(y)\leq A_2 r^2$, $\abs{h}\leq 1/2$, $\abs{x-y}\leq 5/8$, and $128 \abs{x-y}\abs{h}\leq 40$.
Thus from Lemma \ref{l:local-measure-estimate} we see that
\begin{equation}\label{e:Lem4.1LocalGoodSetAtySize}
	\abs{E_y} = |\{z \in B_{c_0r}(y) : u(z,t) < A_2 + 43 \}| > \delta |B_{c_0r}|.
\end{equation}
Here $\del$ is a constant which depends on dimension and the $\mu$ from Lemma \ref{l:local-measure-estimate}.
We note that we use the $r^2$ growth of $\del_h u$ from Lemma \ref{l:local-measure-estimate} in a very rough fashion at this step.   The importance of the $r^2$ comes later, in relationship to an upper bound on $\abs{m(B_r)}$.  We also note that we have used the ball $B_{c_0r}$ instead of $B_r$.  At this stage, both balls have the same estimate regarding the growth of $u$ on a universal proportion of the set.  However, only $B_{c_0r}$ also has the necessary estimate for the size of $m(B_{c_0r})$.  This choice will be further illuminated below.

We need to estimate a set where $u$ is not too large, and given the choice of $E_y$ above, we see that a good candidate is

\begin{equation*}
	NL := \Union_{y\in m^{-1}(G)} E_y.
\end{equation*}
Thanks to (\ref{e:Lem4.1LocalGoodSetAtyDef}) and (\ref{e:Lem4.1LocalGoodSetAtySize}) the measure of $NL$ can be equivalently estimated via the size of
\begin{equation*}
	NLB := \Union_{y\in m^{-1}(G)} B_{c_0r(y)}(y),
\end{equation*}
where $B_{c_0r(y)}(y)$ is the good ball given in Lemma \ref{l:local-measure-estimate}.
Therefore, the only question is whether or not the set, $NLB$, has a measure which is comparable to $B_1$.

If $\{B_j\}$ is a Vitali sub-covering of the collection $\{B_{r(y)}(y)\}_{y\in m^{-1}(G)}$, then we have 
\begin{equation*}
	\Union_j 5 B_j \supset m^{-1}(G), 
\end{equation*}
and hence
\begin{equation*}
	m(\Union_j 5 B_j) \supset m(m^{-1}(G)).
\end{equation*} 
Also by subadditivity, we have that
\begin{equation*}
	\abs{m(\Union_j B_j)} \leq \sum_j \abs{m(B_j)}.
\end{equation*}
In order to conclude, it would suffice to know (and it is true by Lemma \ref{l:mBr}) that 
\begin{equation}\label{e:Lem4.1SizeOfmBr}
	\abs{m(B_j)} \leq C_3\abs{B_j},
\end{equation}
which allows us to compare $\abs{NLB}$ back to $\abs{G}$.  Indeed, the choice to use $B_{c_0r(y)}(y)$ was motivated entirely by this requirement, and Lemma \ref{l:mBr} gives (\ref{e:Lem4.1SizeOfmBr}) via the result of Lemma \ref{l:local-measure-estimate} and the choice of $c_0r(y)$.

We will use the fact that $m$ maps onto $G$ as well as that by construction of the subcover $\{B_j\}$, $m^{-1}(G)$ is contained in its union.  Thus we see that 
\begin{align*}
	G = m(m^{-1}(G)) \subset m(\Union_j B_j) = \Union_j m(B_j),
\end{align*}
and hence by the choice of $c_0r(y)$ and definition of $E_y$ with Lemmas \ref{l:local-measure-estimate} and \ref{l:mBr}, it holds 
\begin{align*}
	\abs{G} \leq \abs{\Union_j m(B_j)}\leq \sum_j \abs{m(B_j)}\leq \sum_j C_3\abs{B_j}\leq  \sum_j \frac{C_3}{\del}\abs{E_{y_j}}.
\end{align*}
Since the $B_j$ were chosen to be disjoint, then also are the corresponding $E_{y_j}$, and so we can conclude
\begin{align*}
	\abs{NL}\geq \abs{\Union_j E_{y_j}} = \sum_j \abs{E_{y_j}}\geq \frac{\del}{C_3}\abs{G}\geq \frac{\del\del_1}{C_3}.
\end{align*}
This gives the result of Lemma \ref{l:growth-lemma}.
\end{proof}

\section{A special barrier function}
\label{sec:Barrier}

This section is concerned with the construction of a barrier function which is essential for all of the results regarding regularity of parabolic (and elliptic) equations in non-divergence form.  In principle, one would expect our construction to be similar to the one presented in \cite[Lemma 4.2]{lara2014regularity}, but this is not actually the case.  We deviate in some significant respects due to the additional generality allowed by assumptions (A2) and (A3).  In this regard our construction is more accurately described as a parabolic version of the barrier from \cite[Section 5]{rang2013h}, where similar lower bounds on only small sets were allowed.  Significant detail is required to carry over the ideas from \cite[Section 5]{rang2013h} to the parabolic setting.  These additional difficulties involved in the construction of the barrier are in fact also related to the conditions under which the Harnack inequality fails for equations such as (\ref{e:main}).

Because of the relative strength of the terms $\abs{\grad p}$ and $M^-p$ under rescaling, it is necessary to break the construction of the special barrier function into two cases: one with $\al\geq1$ and the other with $\al<1$.  For the second case, we must remove the gradient term from the equation.

\subsection{The main lemmas and the barrier}\label{sec:BarrierMainPart}

\begin{lemma} \label{l:barrierAlBigger1}
Let $\al\in[1,2)$ and suppose $r \in (0,1)$ is given. There exists $\eps_0 > 0$, $q_0>0$ and a function $p : \R^d \times (0, \infty) \to \R$ such that for all $\alpha \geq 1$, 
\begin{align}
 p_t + C_0 |\grad p| - M^- p \leq 0 &\text{ in } \Big(B_1 \times (0,\infty)\Big) \setminus \Big(B_r \times (0,r^\al]\Big),  \label{e:barrier1}\\
p \leq 1 &\text{ in } B_r \times (0,r^\al], \label{e:barrier2}\\
p \leq 0 &\text{ in } \Big( \R^d \setminus B_1 \Big) \times (0,\infty) \text{ and } \Big(\R^d\setminus B_r\Big) \times \{0\}, \label{e:barrier3}\\
p \geq \eps_0 r^{q_0} e^{-C_5 (T-r^\alpha)} &\text{ in } B_{3/4} \times [r^\alpha,T], \label{e:barrier4}
\end{align}
The constant $\eps_0$ and $q_0$ depend only on $\lambda$, $\Lambda$, $\mu$, $C_0$, $\alpha_0$ and dimension.
\end{lemma}

\begin{lemma} \label{l:barrierAlSmaller1}
Let $\al\in[\al_0,2)$ and suppose $r \in (0,1)$ is given. Then the same statement of Lemma \ref{l:barrierAlBigger1} remains true except (\ref{e:barrier1}) is replaced by
\begin{align}
 p_t - M^- p \leq 0 &\text{ in } \Big(B_1 \times (0,\infty)\Big) \setminus \Big(B_r \times (0,r^\al]\Big),  \label{e:barrier1SmallAl}
\end{align}
\end{lemma}

\begin{remark}
	We note that the same constants $\eps_0$ and $q_0$ can be chosen to work for both Lemmas \ref{l:barrierAlBigger1} and \ref{l:barrierAlSmaller1}.
\end{remark}

\begin{remark}
	 The existence of the barrier is closely related to uniform estimates on hitting times of a Markov process which are crucial to the proofs of weak Harnack inequality and H\"older regularity in the probabilistic framework.  These hitting time estimates appear in the original work of Krylov-Safonov in \cite{krylov1980certain}, \cite{krylov1979estimate}, and they have become a standard technique in the Probability literature (see the presentation in e.g. the lecture notes \cite{Bass2004-PDEfromProbLectNotes}).  In other
contexts, there exists an explicit barrier and this lemma looks deceivingly simple. For nonlocal equations whose kernels are allowed to vanish, this step is in fact highly non trivial. Lemmas \ref{l:barrierAlBigger1} and \ref{l:barrierAlSmaller1} have a probabilistic interpretation as the lower bound for the probability of the process to hit a ball between time 0 and $r^\alpha$.
\end{remark}

The strategy for this construction is to start with a yet to be determined function, $\Phi$ supported in $B_1$, and rescale $\Phi$ on the time interval $t\in(0,r^\al)$ as 

\begin{equation}\label{e:BarrierPdefViaPhi}
p(x,t) = t^{-q_0}\Phi(\frac{rx}{ t^{1/\al}}),
\end{equation}
and then to use for $t\in(r^\al,\infty)$
\begin{equation}\label{e:BarrierPdefForLargerT}
	p(x,t) = e^{-C_5(t-r^\al)}p(x,r^\al) = e^{-C_5(t-r^\al)} r^{-\al q_0}\Phi(x).
\end{equation}
The choice of $rx/t^{1/\al}$ is to make sure that $p$ will be positive for all $\abs{x}<1$ when $t\geq r^\al$. The constants $q_0$ and $C_5$ are there to force the subsolution property in the regions where $M^-p$ cannot be made to be as large as we like.  We now make some initial computations to illuminate our subsequent choices (note the use of Lemma \ref{l:M+Properties}): 

\begin{align}
	p_t &= -q_0t^{-q_0-1}\Phi(\frac{rx}{t^{1/\al}}) - \frac{1}{\al}t^{-q_0-1/\al-1}\grad\Phi(\frac{rx}{t^{1/\al}})\cdot rx\\
	\grad p &= rt^{-q_0-1/\al}\grad\Phi(\frac{rx}{t^{1/\al}})\\
	M^-p &= t^{-q_0-1} r^\al M^-\Phi(\frac{rx}{t^{1/\al}})
\end{align}
We want to satisfy (\ref{e:barrier1}), which then can be transformed to the new goal (at least for $t\in(0,r^\al)$)
\begin{equation}
	t^{-q_0-1}\left( -q_0 \Phi(\frac{rx}{t^{1/\al}}) - \frac{1}{\al}t^{-1/\al}\grad\Phi(\frac{rx}{t^{1/\al}})\cdot rx + rt^{1-1/\al} C_0\abs{\grad\Phi(\frac{rx}{t^{1/\al}})} 
	-r^\al M^-\Phi(\frac{rx}{t^{1/\al}})  
	\right)
	 \leq 0 .
\end{equation}
Switching out variables 
\[
z=\frac{rx}{t^{1/\al}},
\]
we want for an appropriate set of $z$
\begin{equation}\label{e:PhiGoal}
	t^{-q_0-1}\left( -q_0 \Phi(z) - \frac{1}{\al}\grad\Phi(z)\cdot z + r t^{1-1/\al} C_0\abs{\grad\Phi(z)} 
	-r^\al M^-\Phi(z)  
	\right)
	 \leq 0 .
\end{equation}

We can now turn to the requirement for $p$ to satisfy \eqref{e:barrier1} when $t\geq r^\al$.  The computations are similar to the case of $t\in[0,r^\al]$.  Using (\ref{e:BarrierPdefViaPhi}),

\begin{align*}
	p_t &= -C_5 e^{-C_5(t-r^\al)} r^{-\al q_0} \Phi(x)\\
	\grad p &= r^{-\al q_0} e^{-C_5(t-r^\al)} \grad \Phi(x)\\
	M^-p &= r^{-\al q_0} e^{-C_5(t-r^\al)} M^-\Phi(x).
\end{align*}
Then the goal (\ref{e:PhiGoal}) becomes
\begin{equation}\label{e:BarrierTLargeComputation}
	e^{-C_5(t-r^\al)}r^{-\al q_0}\left( -C_5\Phi(x) + C_0\abs{\grad\Phi(x)} -  M^-\Phi(x)    \right)  \leq 0.
\end{equation}

The function $\Phi$ and subsequently $p$ will  be built in a many-staged process.  One of the key components is a special bump function which acts a a barrier in the stationary setting.  This construction proceeds similarly to that of \cite{rang2013h}, and we would like to point out that there, just as here, there are significant challenges for this construction due to the generality of the lower bound assumption in (\ref{e:assumption-below}) (cf. the bump function in \cite{caffarelli2009regularity} where the lower bound on $K$ holds globally).    We start with a two parameter family of auxiliary functions

\begin{equation*}
	b_{\gam,q}(y) = \hat b(\abs{y})
\end{equation*}
and
\begin{equation}\label{SpecialEq:fHatDef}
	\hat b(r)=
	\begin{cases}
		r^{-q}\ &\text{if}\ r\geq 1-\frac{c_1}{2}\\
		m_{\gam,q}(r)\ &\text{if}\ 1-c_1\leq r \leq 1-\frac{c_1}{2}\\
		\gam^{-q}\ &\text{if}\ r\leq 1-c_1,
	\end{cases}
\end{equation}
with $m_{\gam,q}$ smooth and monotonically decreasing (so there will be a restriction between $\gam$ and $c_1$ both being small enough), and without loss of generality $m_{\gam,q}$ will be such that    
\[
b_{\gam,q}(y) \geq \min\{\gam^{-q},\abs{y}^{-q}\}\ \ \text{for all}\ y\in\R^d.
\]

\begin{center}
\setlength{\unitlength}{1in} 
\begin{picture}(4.86111,2)
\put(0,0){\includegraphics[height=2in]{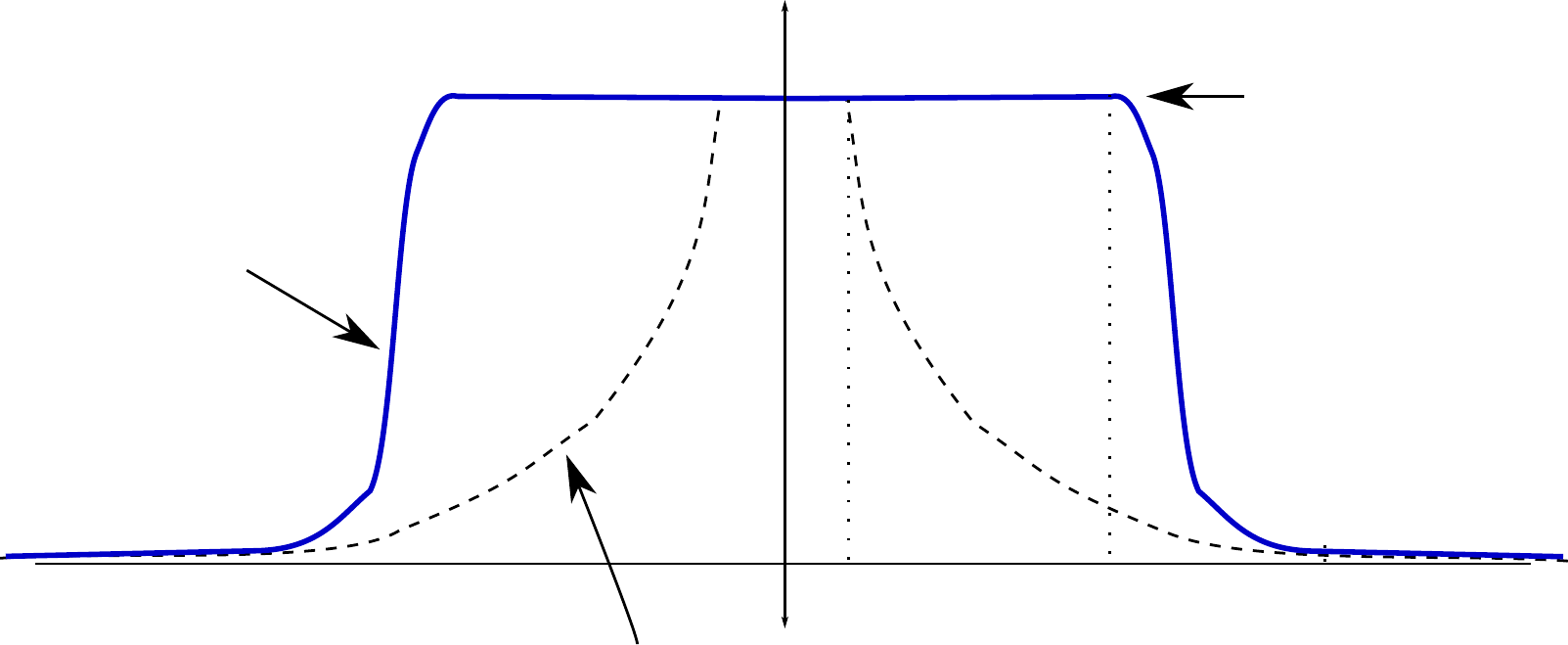}}
\put(0.507,1.22){$b_{\gamma,q}$}
\put(1.99,-0.05){$|y|^{-q}$}
\put(3.88,1.68){$\gamma^{-q}$}
\put(2.58,0.111){$\gamma$}
\put(3.26,0.1){$1-c_1$}
\put(3.94,0.1){$1-c_1/2$}
\end{picture}
\end{center}

The key part of the construction is that there are choices of $\gam$ and $q$ which make $b$ a subsolution in a given small strip (and a subsequent truncation allows the equation to hold in a large set).  We state this result for the choices of $\gam$ and $q$, and then we will prove it in Section \ref{sec:BarrierTechnicalDetails}
\begin{lemma}\label{SpecialLem:TheGoodB}
	Let $C>0$ be given.  Then there exist a small constant $c_1$ and choices of $\gam_1$ and $q_1$ (depending on $C$ plus all other universal objects) such that
	\begin{equation}\label{e:TheGoodB}
	M^-b_{\gam_1,q_1}(x) \geq Cq_1 \abs{x}^{-q_1-\al} \ \ \ \text{for all}\ \ 1-c_1/2 \leq \abs{x}\leq 1,
\end{equation}
	for all $\al\in(\al_0,2)$.  $c_1$ depends on the lower bound of $K$ in (\ref{e:assumption-below}).
\end{lemma}

\begin{remark}
Lemma \ref{SpecialLem:TheGoodB} provides a sub-solution to a stationary problem. It is a generalized version of Corollary 9.2 in \cite{caffarelli2009regularity}, Lemma 3.10 in \cite{BjCaFi-2012NonlocalGradDepOps}, and Lemmas 5.2 and 5.3 in \cite{rang2013h} to the more general class of kernels in this article.
\end{remark}

Now that we know the details of and equation for $b=b_{\gam,q}$, we will continue the calculations which will be useful to construct $p$.  For the following, we assume that $1-c_1/2\leq \abs{z} \leq 1$.  We also note that $\gam_1$, $q_1$, and $C$ will be determined subsequently.

\begin{align}
	b(z) = b_{\gam_1,q_1}(z) &= \abs{z}^{-q_1}\ \text{if}\  1 - c_1/2 < \abs{z} \label{e:Bump1} \\
	\grad b(z) &= -q_1 z \abs{z}^{-q_1-2} \label{e:Bump2}\\
	-\frac{1}{\al} \grad b(z)\cdot z &= \frac{1}{\al}q_1\abs{z}^{-q_1} \label{e:Bump3}\\
	C_0 \abs{\grad b(z)} &= C_0 q_1 \abs{z}^{-q_1-1} \label{e:Bump4}\\
	-M^-b(z) &\leq -Cq_1\abs{z}^{-q_1-\al}. \label{e:Bump5} 
\end{align}

Now that we have sorted out the details regarding $b_{\gam_1,q_1}$, we can proceed with the proof of Lemma \ref{l:barrierAlBigger1}. Some complications arise from the need to satisfy the boundary conditions in (\ref{e:barrier3}).

We will give the proof of Lemma \ref{l:barrierAlBigger1} and then afterwards indicate the few steps which are modified to prove Lemma \ref{l:barrierAlSmaller1}.

\begin{proof}[Proof of Lemma \ref{l:barrierAlBigger1}]
We proceed with defining $p$ in terms of $\Phi$ as described in \eqref{e:BarrierPdefViaPhi} and \eqref{e:BarrierPdefForLargerT}. Note that this construction gives a function $p$ which is unbounded around the origin $(0,0)$. To fix that, at the end of the proof, we have an extra truncation step.

	In order to satisfy the boundary conditions (\ref{e:barrier3}), $\Phi$ will be the following truncated version of $b_{\gam,q}$,	\begin{equation*}
		 \Phi(z) = \max\left\{  b_{\gam,q}(z) - b_{\gam,q}(e_1), 0     \right\}.
	\end{equation*}
This function $\Phi$ is zero outside of $B_1$ and strictly positive inside $B_1$.
The properties of the function $b$ will be used to make the value of $M^- \Phi$ large in $B_1 \setminus B_{1-c_2/2}$.

Recall the variable $z$,
	\begin{equation}\label{e:BarrierProofTXScale}
		z=\frac{rx}{t^{1/\al}}.
	\end{equation}
We need to verify \eqref{e:PhiGoal} and \eqref{e:BarrierTLargeComputation} in order to account for the regions $t\in[0,r^\al]$ and $t\in (r^\al,\infty)$.  We will need to select parameters and constants to work for both ranges of $t$.  But we note that all of the parameters are such that they can be chosen to satisfy both conditions simultaneously.

\noindent
\textbf{Part 1, $t\in[0,r^\al]$}.

Note the following relations, for $z \in B_1$
\begin{align*}
\grad \Phi(z) &= \grad b(z), \\
M^- \Phi(z) &\geq M^- b(z).
\end{align*}
We need to find parameters so that \eqref{e:PhiGoal} holds. The computation will be different in the three regions $|z| \leq 1-c_1/2$, $1-c_1/2 < |z| < 1$, and $|z| \geq 1$.

Replacing \eqref{e:Bump2}, \eqref{e:Bump3}, \eqref{e:Bump4} and \eqref{e:Bump5} in the left hand side of \eqref{e:PhiGoal}, we get
\begin{equation} \label{e:step1goal} \begin{aligned}
-q_0 \Phi(z) - \frac 1 \alpha \grad \Phi(z) \cdot z & + r t^{1-1/\alpha} C_0 |\grad \Phi(z)| - r^\alpha M^- \Phi(z) \\
&\leq -q_0 \Phi(z) - \frac 1 \alpha \grad b(z) \cdot z + rC_0 |\grad b(z)| - r^\alpha M^- b(z)
\end{aligned}
\end{equation}
For the last inequality, we used that $t^{1-1/\alpha} \leq 1$. This is because $t \leq r^\alpha \leq 1$ and $\alpha \geq 1$. When $\alpha < 1$, the negative power of $t$ cannot be controlled and that is why we assume $C_0 = 0$ in those cases.

When $1-c_1/2 < |z| < 1$, we can ignore $-q_0b(z)$, and instead focus on 
\begin{equation}\label{e:barrierChooseBumpParams}
	- \frac{1}{\al}\grad b(z)\cdot z + r C_0\abs{\grad b(z)} - r^\al M^-b(z) \leq 0.
\end{equation}
In light of (\ref{e:Bump3}), (\ref{e:Bump4}), (\ref{e:Bump5}), it will suffice to choose $b$ so that
\begin{equation*}
	\frac{1}{\al}q_1\abs{z}^{-q_1}  + r C_0 q_1 \abs{z}^{-q_1-1} -Cq_1 r^\al\abs{z}^{-q_1-\al} \leq 0,
\end{equation*}
or more succinctly
\begin{equation}\label{e:BumpChoosingC}
	q_1 \abs{z}^{-q_1}\left( \frac{1}{\al}  + r C_0 \abs{z}^{-1} -C r^\al \abs{z}^{-\al} \right) \leq 0.
\end{equation}
After $C$ is chosen to obtain (\ref{e:BumpChoosingC}) (recall $\abs{z}\leq 1$), then $b=b_{\gam_1,q_1}$ can be fixed by Lemma \ref{SpecialLem:TheGoodB}.  The resulting $b$ will be smooth and bounded.

Switching now to the set $\abs{z}\leq 1-c_1/2$, \eqref{e:step1goal} then follows from
\begin{equation}\label{e:BumpChoosingQ0}
	\left( -q_0 \Phi(z) - \frac{1}{\al}\grad b(z)\cdot z  + r C_0 \abs{\grad b(z)} -r^\al M^- b(z)\right) \leq 0.
\end{equation}
The function $\Phi$ is strictly positive in $B_1$ and in particular it is bounded below by a positive constant in $B_{1-c_1/2}$. Since $C$, $\gam_1$, $q_1$ have all been fixed and all of the terms are bounded, we can then choose $q_0$ large enough so that (\ref{e:BumpChoosingQ0}) will also hold. 

We are only left with the case $|z| \geq 1$. Note that because of the angle singularity of the function $\Phi$ on $|z|=1$, we cannot touch the function $\Phi$ from above with any smooth function at those points. Therefore, the points $|z|=1$ play no role in $\Phi$ satisfying \eqref{e:PhiGoal} in the viscosity sense. If $|z|>1$, then $\Phi(z) = |\grad \Phi(z)| = 0$ and $M^- \Phi(z) \geq 0$ because $z$ will be at a global minimum of $\Phi$, and so \eqref{e:PhiGoal} trivially holds.

\medskip

\noindent
\textbf{Part 2, $t\in(r^\al,\infty)$.}

We now need to make sure \eqref{e:BarrierTLargeComputation} holds. The procedure is similar to the first part.

In the region $1-c_1/2 < |x| < 1$, using \eqref{e:Bump4} and \eqref{e:Bump5}, we get
\[ -C_5 \Phi(x) + C_0 |\grad \Phi(x)| - M^- \Phi(x) = -C_5 \Phi(x) + C_0 q_1 |z|^{-q_1-1} - C q_1 |z|^{-q_1-\alpha}.\]
We ignore the term $-C_5 \Phi(x) \leq 0$ and use
\[ -C_5 \Phi(x) + C_0 |\grad \Phi(x)| - M^- \Phi(x) \leq q_1 |x|^{-q_1} \left( C_0 |x|^{-1} - C |x|^{-\alpha} \right) \leq 0.\]
The last inequality holds provided that we choose $C$ large enough (which we can be done by choosing appropriate values of $\gamma$ and $q$ from Lemma \ref{SpecialLem:TheGoodB}).

In the region $|x| < 1-c_1/2$, we use that $b$ (note that $\gamma$ and $q$ are fixed in the previous step) is a given smooth function and $\Phi(x) \geq |1-c_1/2|^{-q} - 1 > 0$. Therefore, picking a large enough $C_5$, we can make \eqref{e:BarrierTLargeComputation} hold.

If $|x| \geq 1$, then the equation holds just as in the first part of this proof, owing the the fact that $z$ will be at a global minimum of $\Phi$.
Note that the constant $C$ which we use for picking $\gamma_1$ and $q_1$ in Lemma \ref{SpecialLem:TheGoodB} need to be large enough to satisfy the requirements of both part 1 ($t \in [0,r^\alpha]$) and part 2 ($t > r^\alpha$) of this proof.

\noindent \textbf{Part 3 - The truncation step.}

Now there is one last step of truncation.  This arises because at this stage, the function $t^{-q_0}\Phi(rx/t^{1/\al})$ has a singularity at $x=0$ and $t\to0$, which of course violates requirement (\ref{e:barrier2}).
 
We define the function 
\[
\tilde p(x,t) := t^{-q_0}\Phi(\frac{rx}{ t^{1/\al}}).
\]
and $p$ will be defined as a truncation of $\tilde p$ to be compatible with (\ref{e:barrier2}). Importantly in this truncation we need to not destroy the equation satisfied by our choice of $\tilde p$ outside of $B_r\times[0,r^\al]$.  That means that we should only truncate at a small enough $t$ so that the support of $\tilde p(\cdot,t)$ is contained in $B_r$.  This way, for such $x$ outside of $B_r$ the desired equation is trivially satisfied because the equation will be evaluated where $\tilde p_t =0$ and $\tilde p(x,t)=0$, which is the global minimum for $\tilde p$, giving $\grad \tilde p=0$ and $M^-\tilde p\geq 0$.  Given the scaling $z=rx/t^{1/\al}$ and that the support of $\Phi$ is in $B_1$ we see that a convenient choice for truncation will be when the graph of $t=(r\abs{x})^{\al}$ intersects the line $|x|=r$, hence at $t=r^{2\al}$.

Accordingly, we define (note for each $t$, $\tilde p$ has its max at $x=0$)
\begin{align*}
p(x,t) &= \frac{\min\{\tilde p(x,t),  \tilde p(0,r^{2\al}) \}}{ \tilde p(0,r^{2\al}) }\\
&= \left( r^{-2\al q_0} \Phi(0)   \right)^{-1} \min\{\tilde p(x,t),  r^{-2\al q_0}\Phi(0) \}.
\end{align*} 
This now gives a complete description of $p$ for $t$ in both $(0,r^\al)$ and $[r^\al,\infty)$ via (\ref{e:BarrierPdefViaPhi}) and (\ref{e:BarrierPdefForLargerT}) respectively.  

The inequality \eqref{e:barrier4} follows by a direct inspection using the expression  \eqref{e:BarrierPdefForLargerT} for $\tilde p$. We get that for $t > r^\alpha$ and $|x| \leq 3/4$,
\[ p(x,t) = \left( r^{-2\al q_0} \Phi(0)   \right)^{-1} e^{-C_5(t - r^\alpha)} r^{-\alpha q_0} \Phi(x) \geq r^{\alpha q_0} e^{-C_5(t - r^\alpha)} \min_{B_{3/4}} \Phi.\]

We note that the truncation expression has shown that the choice of $q$ for the lower bound requirement in (\ref{e:barrier4}) will be $q=\al q_0$. The choice of radius $3/4$ in \eqref{e:barrier4} is irrelevant, since a similar lower bound would hold if $3/4$ is replaced by any other number smaller than one.

This completes the proof of Lemma \ref{l:barrierAlBigger1}.
\end{proof}

We now mention where the proof of Lemma \ref{l:barrierAlSmaller1} deviates from the previous one.

\begin{proof}[Proof of Lemma \ref{l:barrierAlSmaller1}]
	One needs to go back and remove the term $C_0\abs{\grad p}$ from all of the calculations.  Note this was the only term affected by the factor $t^{1-\alpha/2}$ which would be unbounded if $\alpha<1$.
\end{proof}

\subsection{The proof of Lemma \ref{SpecialLem:TheGoodB} }\label{sec:BarrierTechnicalDetails}

Lemma \ref{SpecialLem:TheGoodB} will be attained in two stages, Lemmas \ref{SpecialLem:PLarge} and \ref{SpecialLem:GamSmall}.  First we develop some auxiliary results related to $b$.  First we make a useful observation about the behavior of $\del_h b$.

\begin{lemma}\label{SpecialLem:DeltaYCalc}
	Assume $\al\in[1,2)$. If $b=b_{\gam,q}$ is as in (\ref{SpecialEq:fHatDef}), then for some $r_0$ universal and $C(q)$ so that for $\abs{h}\leq r_0$   and $1-c_1/2 < |x| < 1$,
	\[
	\del_h b(x) \geq -q\frac{\abs{h}^2}{|x|^{q+2}} + q(q+2)\frac{(h_1)^2}{|x|^{q+2}} - C(q)\abs{h}^3,
	\]	
	(this is only relevant-- and only invoked-- for $\al>1$, otherwise we would use a different expansion for $\al<1$).
\end{lemma}

\begin{proof}
	Taylor's theorem.  Note that $h$ is restricted to be in a small set, $B_{r_0}$, and so actually $b(x)=\abs{x}^{-q}$ and $b(x+h) \geq |x+h|^{-q}$.
\end{proof}
The next lemma that says ours assumptions allow that for all $r\leq r_1$, $A_r$ intersects annuli centered at $-e_1$ in a uniformly non-trivial fashion.  This feature is essential to be able to utilize the lower bounds on $K$ in (\ref{e:assumption-below}).

\begin{lemma}\label{SpecialLem:ArHitting}

	There exist constants $c_1$, $c_2$ and $r_1$ (all small), so that:
	\begin{itemize}
		\item[(i)] For any $x$ so that $1-c_1 < |x| < 1$,
	\[
	\left| 
	A_{r_1} \intersect B_{1-c_1}(-x)
	\right|
	\geq \frac{\mu}{4} \abs{ B_{2r_1}\setminus B_{r_1}},
	\]
	
	\item[(ii)] for all $r$
		\[
		\abs{A_r \intersect \{h: (h_1)^2\geq c_2\abs{h}^2  \} }\ \geq \frac{\mu}{2}\abs{B_{2r}\setminus B_r}.
		\] 
	\end{itemize}
\end{lemma}

\begin{proof}
	We first note that by symmetry of $A_r$
	\begin{equation}\label{e:ArHittingLeftHalfSpace}
	\left| A_r\intersect \left(B_{2r}\setminus B_r\right) \intersect \{h : h \cdot x \leq 0\} \right| \geq \frac{\mu}{2} \abs{B_{2r}\setminus B_r}.
\end{equation}
	Now we will establish (i).  We first choose $r_1$ small enough so that 
	\[
	\abs{\left( (B_{2r_1}\setminus B_{r_1}) \intersect \{h : h \cdot x \leq 0\}  \right) \setminus B_{|x|}(-x)} \leq \frac{\mu}{8}\abs{(B_{2r_1}\setminus B_{r_1})\intersect\{h : h\cdot x\leq 0\}}.
	\]
	Note that this choice of $r_1$ can be done uniformly for all $1-c_1 < |x| < 1$.
	
	Let us define the failed set where $A_r$ cannot reach $B_{1-c_1}(-x)$ as
	\[
	F:= \left( (B_{2r_1}\setminus B_{r_1}) \intersect \{h: h\cdot x\leq 0\}  \right) \setminus B_{1-c_1}(-x)
	\]
	With $r_1$ fixed, we can choose $c_1$ small enough so that 
	\begin{equation}\label{e:ArHittingBadSetSmall}
	\abs{F}\leq \frac{\mu}{4}\abs{(B_{2r_1}\setminus B_{r_1})\intersect\{h : h\cdot x\leq 0\}},
\end{equation}
	This is possible because
	\begin{align*}
	 |F| &\leq \abs{\left( (B_{2r_1}\setminus B_{r_1}) \intersect \{h : h \cdot x \leq 0\}  \right) \setminus B_{|x|}(-x)} + |B_{|x|} \setminus B_{1-c_1}| \\
	 &\leq \frac{\mu}{8}\abs{(B_{2r_1}\setminus B_{r_1})\intersect\{h : h\cdot x\leq 0\}} + C(1 - (1-c_1)^d).
	\end{align*}
	Finally, combining (\ref{e:ArHittingLeftHalfSpace}) with (\ref{e:ArHittingBadSetSmall}) we obtain (i).
	
	To establish (ii), we note that
	\begin{align*}
	\abs{A_r \intersect \{h: (h_1)^2\geq c_2\abs{h}^2  \} } &\geq |A_r| - |\{h \in B_{2r} \setminus B_r : h_1^2 < c_2 |h|^2\}|,  \\
	&\geq (\mu - C c_2) |B_{2r} \setminus B_r|.
	\end{align*}
	for a universal constant $C$. Thus, we simply take $c_2$ small enough so that $(\mu - C c_2) \geq \mu/2$.
\end{proof}

\begin{note}\label{SpecialNote:GammaOrdering}
	If $\gam_1<\gam_2$ and $q$ is fixed, then for all $y$
	\begin{equation*}
		b_{\gam_1,q}(y)\geq b_{\gam_2,q}(y),
	\end{equation*}
	and the two functions are equal when $\displaystyle \abs{y}\geq 1-c_1/2$, hence 
	\begin{equation*}
		M^{-}b_{\gam_1,q}(x)\geq M^-b_{\gam_2,q}(x),
	\end{equation*}
	for all $\displaystyle \abs{x}\geq 1-c_1/2$.
\end{note}

Next we make the first choice of parameter for $b$.  It is the selection of the exponent, $q$, and it only uses the information about the family $\mathcal K$ for $\al$ very close to $2$.
\begin{lemma}\label{SpecialLem:PLarge}
	Let $\gam\leq\gam_0=1/4$ be fixed.  Let $C>0$ be given.  Then, there exist a $q_1\geq 1$ and an $\al_1$, depending only on $C$, $\gam_0$, $C_0$, $\mu$, $d$, $\lam$, $\Lam$,   such that 
	\begin{equation*}
		M^-b_{\gam,q_1}(x)\geq Cq_1 \abs{x}^{-q_1-\al} \ \ \ \text{for all}\ \ 1-c_1/2 < \abs{x} < 1,
	\end{equation*}
	 for all orders, $\al\in(\al_1,2)$ and for all $\gam\leq \gam_0$.
\end{lemma}
Then once the $q$ has been chosen, we can finish the definition of $b$ by fixing the truncation height, $\gam^{-q}$, to be large enough (so $\gam$ small enough).  This allows to fix one function which satisfies the special subsolution property for all $\al\in[\al_0,2)$.
\begin{lemma}\label{SpecialLem:GamSmall}
	Let $C>0$ and $q_1$ be as in lemma \ref{SpecialLem:PLarge}.  There there exists a  $\gam_1\leq \gam_0=1/4$  such that 
	\begin{equation*}
		M^-b_{\gam_1,q_1}(x)\geq Cq_1 \abs{x}^{-q_1-\al} \ \ \ \text{for all}\ \ 1-c_1/2 < \abs{x} < 1,
	\end{equation*}
	 for all orders, $\al\in(\al_0,\al_1]$.
\end{lemma}

\noindent 
First we give the proof of Lemma \ref{SpecialLem:PLarge}.

\begin{proof}[Proof of Lemma \ref{SpecialLem:PLarge} ]
	Let $x$ be any point such that $1-c_1/2 < |x| < 1$. We begin with a few simplifying observations.  First of all there is no loss of generality in assuming $\al>1$ for this lemma-- indeed the end of the proof culminates with a choice of $\al_1$ which is sufficiently close to $2$ (hence $\del_h b(x)$ uses only one case for $\al>1$).   Second, to simplify notation we drop the $\gam,q$ dependence and denote $b_{\gam,q} = b$.   
	
	To obtain the bound we want, we only need the contribution of $\del_h b(x)$ to $M^-b(x)$ in a small ball, $h\in B_{r_2}$, for some $r_2$ fixed with say $r_2 = \min\{r_0, c_1/2\}$, where $r_0$ originates in Lemma \ref{SpecialLem:DeltaYCalc} and $c_1$ in comes from Lemma \ref{SpecialLem:ArHitting}.  This is because the large curvature of the graph of $b$ in the $h_1$ direction can be used to dominate the integral at the expense of all the other terms.
	
	We also note that for $h\in\R^d\setminus B_{r_2}$, we have
\begin{equation*}
	\del_h b(x) \geq \inf_{h \in \R^d\setminus B_{r_2}}\left( b(x+h) - b(x) - q |x|^{-q-2} x \cdot h \right) \geq -C_q \left(1+ \frac x {|x|} \cdot h \right).
\end{equation*}
Here $C_q = \max\left( q(1-c_1/2)^{-q-1} , (1-c_1/2)^{-q} \right)$.

Therefore, by Lemma \ref{l:KernelIntegrationBounds},  we see that
\begin{equation}\label{SpecialEq:BoundFromBelowOutsideOfBall}
	\int_{\R^d\setminus B_{r_2}} \del_h b(x) K(h) \dd h \geq -(2-\alpha) C_q \Lam \left( \frac{r_2^{-\al}}{\al}  + r_2^{1-\al} \right).
\end{equation}
Furthermore, combining Lemmas \ref{SpecialLem:ArHitting} and \ref{SpecialLem:DeltaYCalc} we see that on each ring, $B_{2^{-k}r}\setminus B_{2^{-k-1}r}$, we can enhance the positive contribution to $M^-f(x)$ by manipulating the term
\begin{equation*}
	\frac{q(q+2)}{|x|^{q+2}} \int_{B_{2^{-k}r}\setminus B_{2^{-k-1}r}} (h_1)^2 K(h)\dd h.
\end{equation*}
By Lemma \ref{SpecialLem:ArHitting} and assumption (A3), we see that
{\allowdisplaybreaks
\begin{align*}
	&\int_{B_{2^{-k}r}\setminus B_{2^{-k-1}r}} (h_1)^2 K(h)\dd h \geq \int_{A_{2^{-k-1}r}} (h_1)^2 K(h)\dd h\\
	&\ \ \geq \int_{A_{2^{-k-1}r}\intersect \{h:(h_1)^2\geq c_2\abs{h}^2\} } c_2\abs{h}^2 K(h)\dd h\\
	&\ \ \geq c_2(2^{-k-1}r)^2\lam(2-\al)(2^{-k-1}r)^{-d-\al} \abs{A_{2^{-k-1}r}\intersect \{h:(h_1)^2\geq c_2\abs{h}^2\} } \\
	&\ \ \geq c_2(2^{-k-1}r)^2\lam(2-\al)(2^{-k-1}r)^{-d-\al} \frac{\mu}{2} \abs{B_{2^{-k}r}\setminus B_{2^{-k-1}r}}\\
	&\ \ =c_2\lam(2-\al)\mu 2 c(d) r^{2-\al}2^{-k(\alpha-2)},
\end{align*}
}where $c(d)$ is a purely dimensional constant that we use temporarily during this proof.
Hence adding up the contribution along all of the rings, we see
\begin{align}\label{SpecialEq:BoundBelowOnSmallBallCumlinating}
\int_{B_{r_2}} (h_1)^2 K(h)\dd h =\sum_{k=0}^\infty \int_{B_{2^{-k}r_2}\setminus B_{2^{-k-1}r_2}} (h_1)^2 K(h)\dd h \geq \left(\lam \mu c_2 c(d)\right) r_2^{2-\al},
\end{align}
where we have collected various dimensional constants into $c(d)$ in such a way that is uniform for $\al\in(0,2)$. Note that
\[ \sum_{k=0}^\infty (2-\alpha) 2^{k(\alpha-2)} = \frac{2-\alpha}{1-2^{\alpha-2}} \leq 2,\]
for all $\alpha \in (1,2)$.

We also estimate the following integral using Assumption (A2).
\begin{equation} \label{SpecialEq:BoundForCubic}  \begin{aligned}
\int_{B_{r_2}} \abs{h}^3 K(h)\dd h & = \sum_{k} \int_{B_{2^{-k}r_2}\setminus B_{2^{-k-1}r_1}} \abs{h}^3 K(h)\dd h\\
& \leq \frac{(2-\al) 2^\al}{1-2^{\al-3}} r_2^{3-\al}  \Lam.
\end{aligned} 
\end{equation}

\noindent

Now we need to put all of the pieces together.  We will use Lemma \ref{SpecialLem:DeltaYCalc} to balance the terms of different orders in both $\abs{h}$ and $q$.  We will be invoking Lemma \ref{l:KernelIntegrationBounds} as well as the bounds from both (\ref{SpecialEq:BoundFromBelowOutsideOfBall}), (\ref{SpecialEq:BoundBelowOnSmallBallCumlinating}) and \eqref{SpecialEq:BoundForCubic}.
\allowdisplaybreaks{
\begin{align}
&\int_{\R^d} \del_h b(x) K(h) \dd h \nonumber\\
&= \int_{B_{r_2}} \del_h b(x) K(h) \dd h + \int_{\R^d\setminus B_{r_2}} \del_h b(x) K(h) \dd h  \nonumber\\
&\geq \frac{q(q+2)}{|x|^{q+2}} \int_{B_{r_2}} (h_1)^2 K(h)\dd h - \frac q {|x|^{q+2}} \int_{B_{r_2}} \abs{h}^2 K(h)\dd h \nonumber \\
&\ \ \ \ \ \ \ \ \ \ \ \ \ \ \ - C(q) \int_{B_{r_2}} \abs{h}^3 K(h)\dd h + \int_{\R^d\setminus B_{r_2}} \del_h f(x) K(h) \dd h   \nonumber\\
&\geq \frac{q}{|x|^{q+2}} \left( (q+2) \left(\lam \mu c_2 c(d)\right)  - C_d\Lam \right) r_2^{2-\al} 
- (2-\alpha) \left( C_q \Lam \left( \frac{r_2^{-\al}}{\al}  + r_2^{1-\al} \right)
				- C(q) \frac{2^\al}{1-2^{\al-3}}  r_2^{3-\al}  \Lam \right).\label{SpecialEq:ChoosingQ0}
\end{align}
}
At this point we note that the first term is the one which does not have the factor $(2-\alpha)$ in front. We will first choose $q$ large to control the sign of this term. Hence we can choose $q=q_1$ large enough, depending only on the given constant $C$ and the universal parameters, so that (recall $C$, with no subscript, was the parameter given in the statement of this lemma and $|x|<1$)
\[
\frac{q}{|x|^{q+2}} \left( (q+2) \left(\lam \mu c_2 c(d)\right)  - C_d\Lam \right) r_2^{2-\al} \geq 3 C q \abs{x}^{-q-\al} r_2^{2-\alpha}.
\]

Once $q_1$ has been fixed, we can now choose $\al_1$ close enough to $2$ so that the rest of the expression in \eqref{SpecialEq:ChoosingQ0} is small.

\begin{equation*}
(2-\alpha) \left( C_q \Lam \left( \frac{r_2^{-\al}}{\al}  + r_2^{1-\al} \right)	- C(q) \frac{2^\al}{1-2^{\al-3}}  r_2^{3-\al}  \Lam \right)
\leq C q_1 \abs{x}^{-q_1-\al} r_2^{2-\alpha}.
\end{equation*}
(Recall that $r_2 = \min\{r_0, c_1/2\}$.)  Thus we have achieved 
\[
\int_{\R^d} \del_h b(x) K(h) \dd h \geq 2 Cq_1\abs{x}^{-q_1-\al} r_2^{2-\alpha}
\]
The chosen value of $\alpha$ is sufficiently close to $2$. We may choose $\alpha$ even closer to $2$ so that $r_2^{2-\alpha} > 1/2$ and
\[
\int_{\R^d} \del_h b(x) K(h) \dd h \geq C q_1 \abs{x}^{-q_1-\al}
\]
Taking an infimum over $K$ yields the result.
\end{proof}

\begin{remark}
The underlying reason why the previous proof works is because if we fix the values of $\Lambda$, $\lambda$ and $\mu$, the following limit holds
\[ \lim_{\alpha \to 2} M^- b(x) = \mathcal M^-_{\tilde \lambda, \tilde \Lambda}(D^2 b(x)),\]
where $\mathcal M^-$ is the classical minimal Pucci operator of order two and $\tilde \lambda$ and $\tilde \Lambda$ are ellipticity constants which depend on $\lambda$, $\Lambda$, $\mu$ and dimension. The proof of this fact goes along the same lines as the proof of Lemma \ref{SpecialLem:PLarge}.  
\end{remark}

\begin{remark}
	We note that the statement and proof of Lemma \ref{SpecialLem:PLarge} here, combined with step 1 of the proof of Lemma \ref{l:barrierAlBigger1}, corrects an error in the construction of the similar barrier used in \cite[Section 5]{rang2013h} where the truncation step should have been done first, not at the end of the construction.
\end{remark}

\noindent
Now we can conclude this section with the proof of Lemma \ref{SpecialLem:GamSmall}.

\begin{proof}[Proof of Lemma \ref{SpecialLem:GamSmall}]
	Let $x$ be any point such that $1-c_1/2 < |x| < 1$. First of all, we note that $q_1$ has been fixed already, so we will drop it from the notation. Since we will be manipulating the choice of $\gam$ to obtain the desired bound on $M^-b_{\gam,q_1}(x)$, it will be convenient to have bounds which transparently do not depend on $\gam$.  Therefore, as above we keep $\gam_0=1/4$ fixed and we will use and auxiliary function to make some of the estimates.  Let $\phi$ be any function in $C^2(\R^d)$ such that 
	\[
	0\leq \phi\leq b_{\gam_0,q_1}\ \ \text{in}\ \R^d,
	\]
	and 
	\[
	\phi(x) = \abs{x}^{-q_1}\ \ \forall\ \abs{x}\geq 1-c_1/2.
	\]
We note that these definitions imply $\norm{\phi}_{C^2}$ can be chosen to be independent of $\gam$ (depending on universal parameters plus $\gam_0$, $q_1$).	

We now estimate the contributions from the positive and negative parts of $(\del_h b(x))^\pm$	separately.  The first estimate below is simply a use of the fact that by construction, $\phi$ touches $b$ from below at $x$, and the second one uses (\ref{SpecialEq:BoundFromBelowOutsideOfBall}).
	\begin{align}
		&\int_{\R^d} \left( \del_h f(x) \right)^- K(h) \dd h  \nonumber \\
		&\ \ \leq \int_{B_{r_1}} C(d)(\norm{\phi}_{C^{1,1}(B_{1/2}(x))}) \abs{h}^2 K(h) \dd h + \int_{\R^d\setminus B_{r_1}} \left( \del_h f(x)\right)^- K(h) \dd h \nonumber \\
		&\ \ \leq C_dC(d)\norm{\phi}_{C^{1,1}(B_{1/2}(x))}\Lam r_1^{2-\al} + C_d\frac{\Lam}{\al} r_1^{-\al} + q_0C_d\Lam r_1^{1-\al}. \label{SpecialEq:LargeGamNegPartBound}
	\end{align}

Now we move to $(\del_h b(x))^+$.   Here we will use Lemma \ref{SpecialLem:ArHitting} part (i), the important feature being that there is \emph{at least one} good ring where $(\del_h b(x))^+$ will see the influence of the value of $b$ on the set $B_{1-c_2}$.    We alert the reader to a strange term in line (\ref{SpecialEq:SmallGamIntermediateStepDelhb})-- below-- which arises simply as a worst case scenario of the three definitions of $\del_h$, and for example if $\al<1$ the term would not even be necessary.  It does not harm the computation, and so we leave it there for any of the possible three cases of $\del_h$ via $\al$.  Finally we note the important feature that we may only integrate on the set $h\in B_{1-c_1}(-x)$, which allows us to avoid the singularity of $K$ at $h=0$.  Also note if $h\in B_{1-c_1}(-x)$, then $\abs{h}\leq 2$. 
{\allowdisplaybreaks
	\begin{align}
		&\int_{\R^d} \left(\del_h f(x)\right)^+ K(h) \dd h \nonumber \\
		&\ \ \geq \int_{A_{r_1}\intersect B_{1-c_1}(-x) } \left(\del_h f(x)\right)^+ K(h) \dd h \nonumber \\
		&\ \ \geq \int_{A_{r_1}\intersect B_{1-c_1}(-x)} (\gam^{-q_1}-|x|^{-q_1})K(h) \dd h - q_1 |x|^{-q_1-1} \int_{B_{1-c_1}(-x)}\abs{h} K(h) \dd h \label{SpecialEq:SmallGamIntermediateStepDelhb} \\
		&\ \ \geq (\gam^{-q_1}-(1-c_1/2)^{-q_1}) (2-\al)\lam \int_{A_{r_1}\intersect B_{1-c_1}(-e_1)} \abs{h}^{-d-\al} \dd h \\
& \qquad \qquad \qquad \qquad \qquad \qquad \qquad	-q_1 (1-c_1/2)^{-q_1-1} \int_{B_{1-c_1}(-e_1)} 2 K(h)\dd h  \nonumber  \\
		&\ \ \geq (\gam^{-q_1}-(1-c_1/2)^{-q_1}) (2-\al)\lam r_1^{-d-\al}\frac{\mu}{4}\abs{B_{2r_1}\setminus B_{r_1}}
					-q_1 (1-c_1/2)^{-q_1-1} (2-\al)C(d,\al_0) \label{SpecialEq:LargeGamPosPartGoodeLine2}
	\end{align}
}

\noindent
We note the use of Lemma \ref{l:KernelIntegrationBounds} (\ref{e:integral-value-tail}) in the transition between the last two lines. 

Recall that the values of $c_1$ and $q_1$ were fixed in Lemmas \ref{SpecialLem:ArHitting} and \ref{SpecialLem:PLarge}. In order to conclude the proof, we see that we can choose $\gam=\gam_1$ large enough that when we add together the contribution from (\ref{SpecialEq:LargeGamNegPartBound}) and  (\ref{SpecialEq:LargeGamPosPartGoodeLine2}) the final estimate becomes greater than $C>q_1$ for all $\al\in(\al_0,\al_1)$.  We note that it is crucial to have $\al\leq \al_1<2$ in this case in order to keep $\al$ uniformly away from $2$, which would cause problems.	
\end{proof}

\section{An estimate in $L^\eps$ -- the weak Harnack inequality}
\label{sec:LEpsilon}
\label{s:Leps}

The purpose of this section is to combine the point-to-measure estimate with the special barrier to prove the $L^\eps$ estimate, also called the \emph{weak Harnack inequality}.

\begin{thm}[The $L^\eps$ estimate] \label{t:leps}
Assume $\alpha \geq \alpha_0 > 0$. Let $u$ be a function such that
\begin{align*}
u &\geq 0 \ \text{ in } \R^d \times [-1,0], \\
u_t + C_0 |\grad u| - M^- u &\leq C \ \text{ in } Q_1, \\
\end{align*}
and for the case, $\alpha < 1$, further assume $C_0=0$.
Then there are constants $C_6$ and $\eps$ such that
\[ \left( \int_{B_{1/4} \times [-1,-2^{-\alpha}]} u^\eps \dx \dd t \right)^{1/\eps} \leq C_6 \left( \inf_{Q_{1/4}} u + C \right).\]
The constants $C_6$ and $\eps$ depend on $\alpha_0$, $\lambda$, $\Lambda$, $C_0$, $d$ and $\mu$.
\end{thm}

Note that the $L^\epsilon$ norm of $u$ is computed in the cylinder $B_{1/4} \times [-1,-2^{-\alpha}]$. This cylinder lies earlier in time than the cylinder $Q_{1/2}$ where the infimum is taken in the right hand side of the inequality. This is natural due to the causality effect of parabolic equations. What should be noted in this case is that, due to the scaling of the equation, the size of these cylinders varies. Indeed, if $\alpha \in (1,2)$, then the time interval $[-1,-2^{-\alpha}]$ is longer than $1/2$ and certainly longer than $[-4^{-\alpha},0]$, which is the time span of $Q_{1/4}$. However, for small values of $\alpha$, the length of $[-1,-2^{-\alpha}]$ becomes arbitrarily small and the time span of $Q_{1/4}$ is almost one. We still have uniform choices of the constants $C$ and $\eps$ because of the assumption $\alpha \geq \alpha_0 >0$.

\begin{center}
\setlength{\unitlength}{1in} 
\begin{figure}[h!b]
        \centering
        \begin{subfigure}{2.5in}
        \begin{picture}(2.5,1)
\put(0,0){\includegraphics[height=1in]{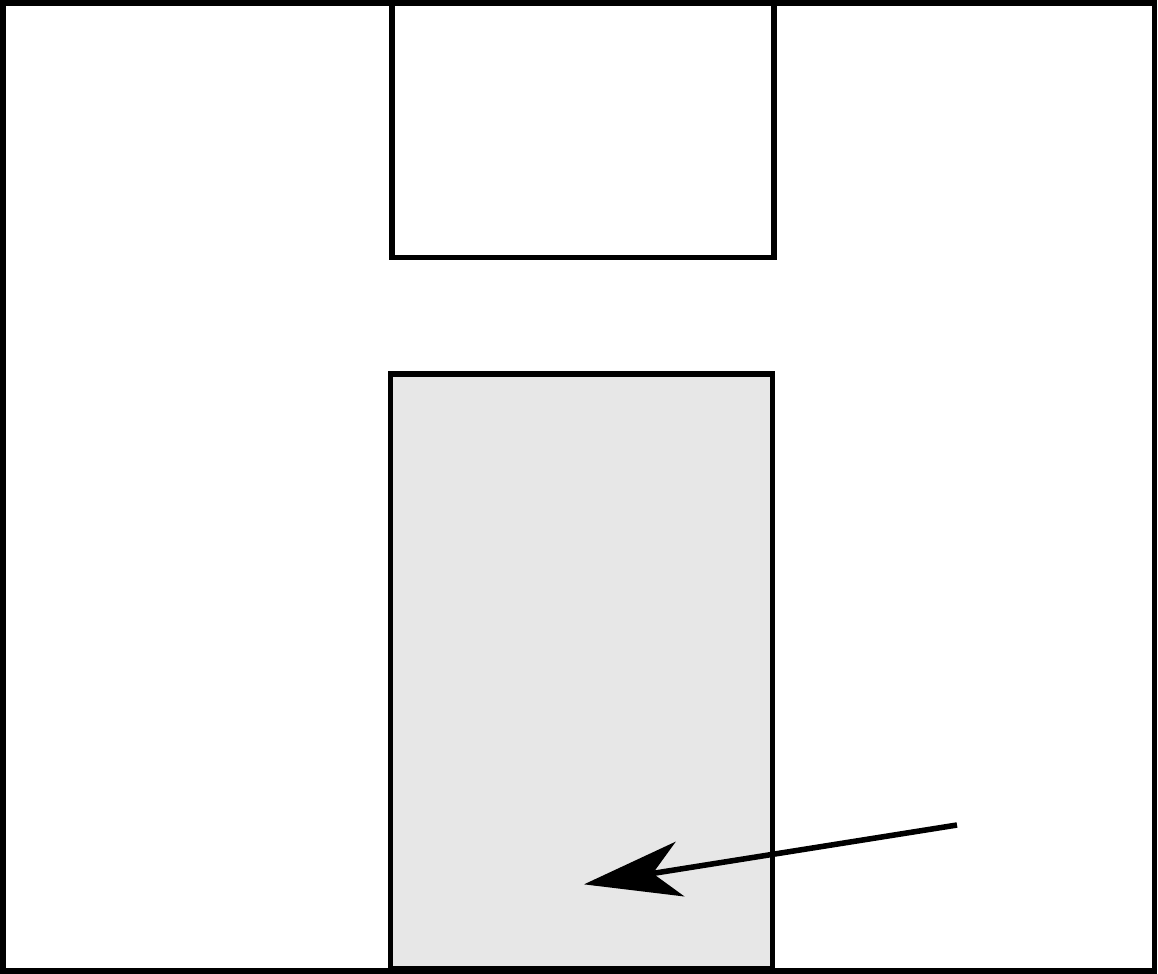}}
\put(0.0278,0.479){$Q_1$}
\put(0.417,0.806){$Q_{1/4}$}
\put(1,0.1){$B_{1/4} \times [-1,-2^{-\alpha}]$}
\end{picture}
\caption{Large $\alpha$}
\end{subfigure}
\begin{subfigure}{2.5in}
\begin{picture}(2.5,1)
\put(0,0){\includegraphics[height=1in]{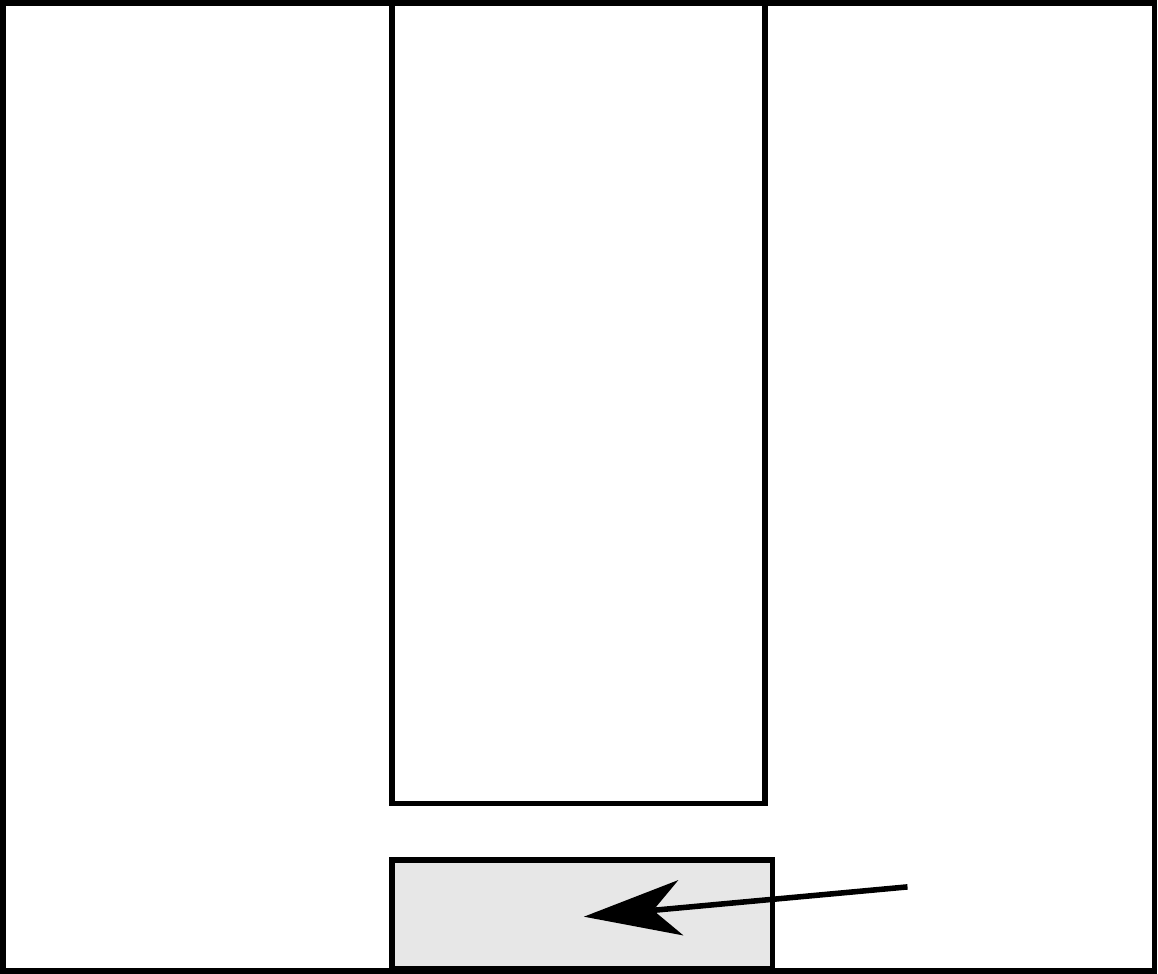}}
\put(0.95,0.06){$B_{1/4} \times [-1,-2^{-\alpha}]$}
\put(0.431,0.799){$Q_{1/4}$}
\put(0.0278,0.479){$Q_1$}
\end{picture}
\caption{Small $\alpha$}
\end{subfigure}
\end{figure}
\end{center}

The basic building block of this proof is Lemma \ref{l:growth-lemma}, which needs to be combined with Lemmas \ref{l:barrierAlBigger1} and \ref{l:barrierAlSmaller1} as well as a covering argument. Since the work of Krylov and Safonov \cite{krylov1980certain}, it is known that these ingredients lead to Theorem \ref{t:leps}. However, there are several ways to organize the proof and there are some subtleties that we want to point out. Thus, we describe the full proof explicitly. We start with some preparatory lemmas.

The following Lemma plays the role of Corollary 4.26 in \cite{imbert2013introduction}, which the reader can compare with Corollary 5.2 in \cite{lara2014regularity}.
Recall the notation $Q_r(x,t) = B_r(x) \times [t-r^\alpha,t]$. We now define a time shift of the cylinder $Q$, which we call $\bar Q^m$. For any positive number $m$, we write $\bar Q^m$ to denote
\[ \bar Q^m = B_r(x) \times (t,t+mr^\alpha).\]
The cylinder $\bar Q^m$ starts exactly where $Q$ ends. Moreover, its time span is enlarged by a factor $m$. Because of the order of causality, the information we have about the solution $u$ in $Q$ propagates to $\bar Q^m$. This is reflected in the following lemma.

\begin{lemma}[Stacked point estimate] \label{l:stacked}
Let $m$ be a positive integer. There exist $\delta_2>0$ and $N>0$ depending only on $\lambda$, $\Lambda$, $d$, $\alpha_0$ and $m$ such that if the following holds for some cylinder $Q = Q_\rho(x_0,t_0) \subset Q_1$,
\begin{align}
u &\geq 0  \text{ in } \R^d \times [-1,0],\\
u_t + C_0 |\grad u| - M^- u &\geq 0 \text{ in } Q_1, \label{e:spe-equation} \\
\abs{ \{u \geq N\} \cap Q_\rho(x_0,t_0) } &\geq  (1-\delta_2) |Q_\rho|, \label{e:spe-measure}\\
B_{2\rho}(x_0) \times [t_0-\rho^\alpha,t_0+m \rho^\alpha] &\subset Q_1,
\end{align}
then $u \geq 1$ in $\bar Q^m = B_\rho(x_0) \times [t_0,t_0+m \rho^\alpha]$.
\end{lemma}

\begin{figure}[hb]
\centering
\setlength{\unitlength}{1in} 
\begin{picture}(1.20833,1)
\put(0,0){\includegraphics[height=1in]{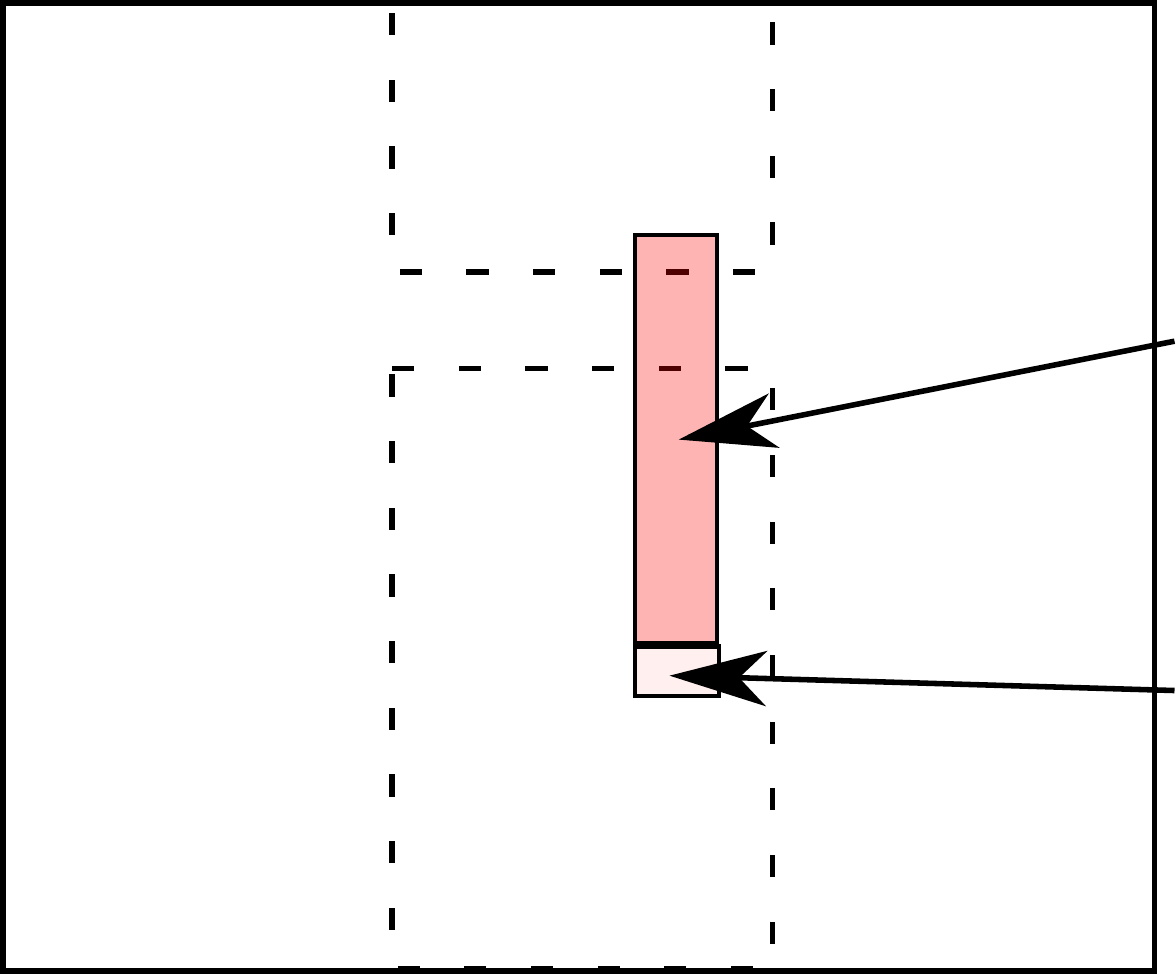}}
\put(1.22,0.229){$Q = Q_\rho(x_0,t_0)$}
\put(1.22,0.591){$\bar Q^m = B_{\rho}(x_0) \times [t_0,t_0+m\rho^\alpha]$}
\put(0.0278,0.479){$Q_1$}
\end{picture}
\caption{The cylinders involved in Lemma \ref{l:stacked}}
\end{figure}
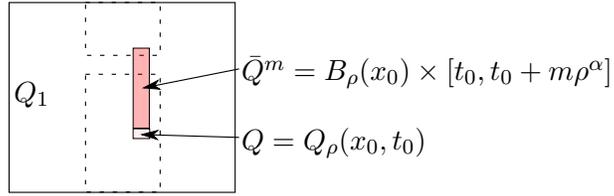

\begin{proof}
Let $\tilde u$ be the scaled function
\[ \tilde u(x,t) = \frac{A_0}{N} u(\rho x + x_0,\rho^\alpha t+t_0),\]
where $A_0$ is the constant from Lemma \ref{l:growth-lemma}.

This function satisfies the same equation \eqref{e:spe-equation}. From our assumption \eqref{e:spe-measure}, we have that $|\{\tilde u > A_0\} \cap Q_1| \geq (1-\delta_2) |Q_1|$. Applying the contrapositive of Lemma \ref{l:growth-lemma}, we obtain that $\tilde u \geq 1$ in $Q_{1/4}$. Thus,
\[ u \geq \frac{N}{A_0}\text{ in } Q_{\rho/4}(x_0,t_0).\]

Recall that $u$ is a supersolution in $Q_1$ and $u \geq 0$ everywhere.  We apply Lemmas \ref{l:barrierAlBigger1} or \ref{l:barrierAlSmaller1} with $r=1/2$ to obtain the subsolution, $p$, and we can compare the functions $\tilde u$ and $p$.  Writing this in terms of $u$ gives 
\[ u(x,t) \geq \frac NM p\left(\frac{(x-x_0)}\rho, \frac{(t-t_0+(\rho/4)^\alpha)}{\rho^\alpha}\right).\]
The conclusion follows from taking $N$ large enough, combined with the lower bound for $p$ given in Lemma \ref{l:barrierAlBigger1}.  
\end{proof}

The point of the previous lemma is that it can be combined with the \emph{crawling ink spots} theorem. This is a covering argument which can be used as an alternative to the Calderon-Zygmund decomposition, and it is close to the original argument by Krylov and Safonov in \cite{krylov1980certain}. It has the cosmetic advantage that it does not use cubes but only balls. Moreover, the Calderon-Zygmund decomposition uses that we can tile the space with cubes, which is only true for $\alpha = 1$. In \cite{lara2014regularity}, this difficulty is overcome by a special tiling with variable scaling which is explained by the beginning of section 4.2. It is a cumbersome construction to define rigorously. The use of the crawling ink spots theorem completely avoids this difficulty.

\begin{thm} [The crawling ink spots] \label{t:ink-spots}
Let $E \subset F \subset B_{1/2} \times \R$. We make the following two assumptions.
\begin{itemize}
\item For every point $(x,t) \in F$, there exists a cylinder $Q \subset B_1 \times \R$ so that $(x,t) \in Q$ and $|E \cap Q| \leq (1-\mu) |Q|$.
\item For every cylinder $Q \subset B_1 \times \R$ such that $|E \cap Q| > (1-\mu) |Q|$, we have $\bar Q^m \subset F$.
\end{itemize}
Then 
\[ |E| \leq \frac{m+1} m (1-c \mu) |F|.\]
Here $c$ is an absolute constant depending on dimension only.
\end{thm}

The proof of Theorem \ref{t:ink-spots} will be presented in the appendix.  The crawling ink spots theorem is used with a value of $m$ sufficiently large so that $\frac{m+1} m (1-c \delta) < 1$. 
In order to prove the $L^\eps$ estimate, we would want to apply Theorem \ref{t:ink-spots} with 
\[E = \{u \geq N^{k+1}\} \cap B_{1/2} \cap (-1,-2^{-\alpha}) \text{ and } F = \{u \geq N^k\} \cap B_{1/2} \cap (-1,-2^{-\alpha}).\]
The problem is that the assumption of Theorem 4.4 is not implied by Lemma \ref{l:stacked} because there is no way to assure that $t+mr^\alpha \leq -2^{-\alpha}$. This is a difficulty which is non existent in the elliptic setting. Because of the time shift in all the point estimates, the conclusion of the crawling ink spots theorem may be spilling outside of the time interval $[-1,-2^{-\alpha}]$. There is no trivial workaround for this.

The purpose of the following lemma is to show that the cylinders $Q_\rho(x_0,\rho_0)$ which satisfy the condition of the crawling ink spots theorem are necessarily small, and consequently the amount of measure that \emph{leaks} outside the cylinder $B_{1/4} \times [-1,-2^{-\alpha}]$ will decay exponentially.

\begin{lemma} \label{l:onlysmallcylinders}
Assume that
\begin{align*}
\inf_{Q_{1/4}} u \leq 1, & \\
u \geq 0 & \text{ in } \R^d \times [-1,0],\\
u_t + C_0 |\grad u| - M^- u \geq 0 &\text{ in } Q_1, \\
\intertext{and that there is a cylinder $Q_\rho(x_0,t_0)$ such that}
Q_\rho(x_0,t_0) &\subset B_{1/4} \times [-1,-2^{-\alpha}],\\
\abs{ \{u \geq N \} \cap Q_\rho(x_0,t_0) } &\geq (1-\delta_2) |Q_\rho|
\end{align*}
Then $\rho < C \, N^{-\gamma}$ for some universal $\gamma>0$ and $C>0$.
\end{lemma}

\begin{proof}
Applying Lemma \ref{l:growth-lemma} rescaled to $Q_\rho(x_0,t_0)$, we obtain that $u \geq N/M$ in $Q_{\rho/4}(x_0,t_0)$.  Just as in the proof of Lemma \ref{l:stacked}, we get
\[ u(x,t) \geq \frac N M p\left( \frac 43 (x-x_0), \left(\frac 43 \right)^\alpha (t-t_0+(\rho/4)^\alpha) \right),\]
where $p$ is the function from Lemmas \ref{l:barrierAlBigger1} or \ref{l:barrierAlSmaller1} with $r = \rho/3$. The reason for the factor $4/3$ is that since $x_0 \in B_{1/4}$, we know that $B_{3/4}(x_0) \subset B_1$.

We have that $x_0 \in B_{1/4}$, $t_0 \in [-1,-2^{-\alpha}]$ and $\rho \leq \min(1/4, (1-2^{-\alpha})^{1/\alpha})$.
Since $\inf_{Q_{1/4}} u \leq 1$, then
\[ \frac MN \geq \inf \left\{ p(x,t) : x \in B_{2/3} \wedge t \in [(3^{-\alpha} (2^\alpha-1) + (\rho/3)^\alpha , (4/3)^\alpha + (\rho/3)^\alpha] \right\} \geq c \rho^q,\]
which holds by (\ref{e:barrier4}) in Lemmas \ref{l:barrierAlBigger1} and \ref{l:barrierAlSmaller1}.
Therefore $\rho < C N^{-\gamma}$, where $\gamma = 1/q$ and $q$ is the exponent from Lemmas \ref{l:barrierAlBigger1} or \ref{l:barrierAlSmaller1}.

\end{proof}

\begin{proof}[Proof of Theorem \ref{t:leps}]
We start by noting that we can assume $C=0$. Otherwise we consider $\tilde u(x,t) = u(x,t) - C t$ instead.
For every positive integer $k$, let \[A_k := \{u > N^k\} \cap (B_{1/4} \times (-1,-2^{-\alpha})).\]
where $N$ is the constant from Lemma \ref{l:stacked}.
We apply Theorem \ref{t:ink-spots} with
\[E = \{u \geq N^{k+1}\} \cap (B_{1/4} \times (-1,-2^{-\alpha})) \text{ and } F = \{u \geq N^k\} \cap (B_{1/4} \times (-1,-2^{-\alpha} + Cm N^{-\gamma \alpha k})),\]
where $C$ and $\gamma$ are the constants from Lemma \ref{l:onlysmallcylinders}.

Let us verify that both assumptions of Theorem \ref{t:ink-spots} are verified. 
The first assumption in Theorem \ref{t:ink-spots} is implied by Lemma \ref{l:onlysmallcylinders} (at least when $N$ and/or $k$ are large). Indeed, any point $(x,t) \in B_{1/4}\times  (-1,-2^{-\alpha} + m N^{-\gamma \alpha k})$ is contained in some cylinder $Q_r(x_0,t_0)$ with large enough $\rho$ so that $\rho > C N^{-k\gamma}$. 
Because of Lemma \ref{l:stacked}, whenever there is a cylinder $Q$ such that $|A_{k+1} \cap Q| \geq (1-\delta) |Q|$, then $\bar Q^m \subset \{ u > N^k\}$. Moreover, because of Lemma \ref{l:onlysmallcylinders}, then the length in time of $\bar Q^m$ is less than $mC N^{-\gamma k}$. Therefore $\bar Q^m \subset F$. Thus, the second assumption of Theorem \ref{t:ink-spots} holds as well.

Note that we allow the result of the crawling ink spots theorem to spill to the time interval $[-2^{-\alpha},-2^{-\alpha}+Cm N^{-\gamma \alpha k}]$.  Therefore,
\[ |A_{k+1}| \leq \frac{m+1} m (1-c \delta) \left( |A_k| + C m N^{-\gamma \alpha k} \right).\]
We fist pick $m$ sufficiently large so that
\[ \frac{m+1} m (1-c \delta) := 1-\mu < 1.\]
Thus, we have
\[ |A_{k+1}| \leq (1-\mu) \left( |A_k| + C m N^{-\gamma \alpha k} \right).\]
This already implies an exponential decay on $|A_k|$, which proves the theorem.
\end{proof}


\section{H\"older continuity of solutions}\label{sec:Holder}

We first state a H\"older continuity for parabolic integral equations without drift. In this case $\alpha \in (0,2)$ can be arbitrarily small, although the estimates depend on its lower bound $\alpha_0$.

\begin{thm}[H\"older estimates without drift] \label{t:holder-no-drift}
Assume $\alpha \geq \alpha_0 > 0$. Let $u$ be a bounded function in $\R^d\times [-1,0]$ such that
\begin{align*}
u_t - M^+ u \leq C &\text{ in } Q_1, \\
u_t - M^- u \geq -C &\text{ in } Q_1.
\end{align*}
then there are constants $C_7$ and $\gamma$, depending on $n$, $\lambda$, $\Lambda$ and $\alpha_0$, such that
\[ \| u \|_{C^\gamma(Q_{1/2})} \leq C_7 \left( \|u\|_{L^\infty(\R^d \times [-1,0])} + C \right).\]
\end{thm}

We can also include a drift term in the equation when $\alpha \geq 1$. This is stated in the next result.

\begin{thm}[H\"older estimates with drift] \label{t:holder-with-drift}
Assume $\alpha \geq 1$. Let $u$ be a bounded function in $\R^d \times [-1,0]$ such that
\begin{align*}
u_t - C_0 |\grad u| - M^+ u \leq C &\text{ in } Q_1, \\
u_t + C_0 |\grad u| - M^- u \geq -C &\text{ in } Q_1.
\end{align*}
then there are constants $C_7$ and $\gamma$, depending on $n$, $\lambda$, $\Lambda$, $C_0$, such that
\[ \| u \|_{C^\gamma(Q_{1/2})} \leq C_7 \left( \|u\|_{L^\infty(\R^d \times [-1,0])} + C \right).\]
\end{thm}

The proofs of these two theorems are essentially the same. The only difference is that when $\alpha \geq 1$ we can include a non zero drift term in Theorem \ref{t:leps}. Because of this, we write the proof only once, for Theorem \ref{t:holder-with-drift}, which applies to both theorems.

\begin{proof}[Proof of Theorem \ref{t:holder-with-drift}]
We start by observing that we can reduce to the case $C \leq \eps_0$ and $\|u\|_{L^\infty} \leq 1/2$ by considering the function
\[ \frac 1 {C/\eps_0 + 2 \|u\|_{L^\infty}} u(x,t).\]
We choose $\eps_0$ sufficiently small, which will be specified below.

Our objective is to prove that for some $\gamma > 0$, which will also be specified below, 
\begin{equation} \label{e:holderaim}  \osc_{Q_r} u \leq 2 r^\gamma,
\end{equation}
for all $r \in (0,1)$. This proves the desired modulus of continuity at the point $(0,0)$. Since there is nothing special about the origin, we obtain the result of the theorem at every point in $Q_{1/2}$ using a standard scaling and translation argument.
Note that since $\|u\|_{L^\infty} \leq 1/2$, we know a priori that \eqref{e:holderaim} holds for all $r < 2^{-1/\gamma}$. We can make this threshold arbitrarily small by choosing a small value of $\gamma$.

In order to prove that \eqref{e:holderaim} holds for all values of $r \in (0,1)$, we use induction. We assume that it holds for all $r \geq 8^{-k}$ and we show that it then holds for all $r \geq 8^{-(k+1)}$. Because of the observation in the previous paragraph, we can guarantee this inequality for the first few values of $k$ by choosing a small value of $\gamma$. Thus, we are left to prove the inductive step.

Let
\[ \tilde u(x,t) = \frac 12 \ \frac 1 { 8^{\gamma(k-1)}} u \left( \frac{8^{-(k-1)}} 2 x, \frac{8^{-\alpha(k-1)}} {2^\alpha} t \right).\]
This function $\tilde u$ is a scaled version of $u$ so that the values of $\tilde u$ in $Q_2$ correspond to the values of $u$ in $Q_{8^{-k+1}}$. Moreover, since \eqref{e:holderaim} holds for $r \geq 8^{-k}$, we have that
\begin{equation} \label{e:al1}
	\osc_{Q_{2r}} \tilde u \leq \min(r^\gamma,1),  
\end{equation}
for all $r \geq 1/8$. 

Since $\osc_{Q_2} \tilde u \leq 1$, then for all $(x,t) \in Q_2$, we have that $\tilde u(x,t) \geq \max_{Q_2} \tilde u - 1/2$ or $\tilde u(x,t) \leq \min_{Q_2} \tilde u + 1/2$. There may be points where both inequalities hold. The important thing is that at least one of the two inequalities holds at every point $(x,t) \in Q_2$. Therefore, one of the two inequalities will hold in at least half of the points (in measure) of the cylinder $B_{1/4} \times [-1,-2^{-\alpha}]$. Without loss of generality, let us assume it is the first of these inequalities the one which holds for most points (a similar argument works otherwise). That is, we have
\[ \left\{ \tilde u \geq \max_{Q_2} \tilde u - \frac 12 \right\} \cap \left( B_{1/4} \times [-1,-2^{-\alpha}] \right) \geq \frac 12 |B_{1/4}| \times (1-2^{-\alpha}).\]

Let $v$ be the truncated function
\[ v(x,t) := \left( \tilde u(x,t) - \max_{Q_2} \tilde u + 1 \right)^+.\]
Note that $v \geq 0$ everywhere and $v = 
\tilde u(x,t) - \max_{Q_2} \tilde u + 1$ in $Q_2$. If $x \notin B_2$ and $t \in [-1,0]$, it can happen that $v(x,t) > \tilde u(x,t) - \max_{Q_2} \tilde u + 1$. We can estimate their difference using \eqref{e:al1}.
\begin{equation} \label{e:al2}  
	v(x,t) - \left(\tilde u(x,t) - \max_{Q_2} \tilde u + 1\right) \leq \osc_{B_{|x|}\times[-1,0]} \tilde u - 1 \leq \left(\frac{|x|}2 \right)^\gamma - 1, \ \text{ for any } x \notin B_2, \ t \in [-1,0].
\end{equation}
Note that for any fixed $R$, the right hand side converges to zero uniformly for $2\leq \abs{x}\leq R$ as $\gamma \to 0$.

Inside $Q_1$, the function $v$ satisfies the following equation
\begin{align*}
v_t + C_0 |\grad v| - M^- v  
&\geq \tilde u_t + C_0\abs{\grad \tilde u} - M^- \tilde u + M^-(\tilde u - v)     \\
&\geq -\eps_0 + M^-(\tilde u - v)\\
&= -\eps_0 + M^-( (\tilde u - \max \tilde u +1) - v)\\
&\geq -\eps_0 - c(\gamma).
\end{align*}
Here $c(\gamma) = -\min_{Q_1} M^-((\tilde u - \max \tilde u +1)-v) = \max_{Q_1} M^+(v-(\tilde u - \max \tilde u +1))$. We can estimate $c(\gamma)$ using \eqref{e:al2} and assumption (A2), because 
\begin{align}
	L(v-(\tilde u - \max \tilde u +1))(x) &= \int_{\R^d} \del_h \left( v-(\tilde u - \max \tilde u +1) \right)(x) K(h)\dd h
	\nonumber \\
	&= \int_{\abs{h}\geq 2} \left( v-(\tilde u - \max \tilde u +1) \right)(h) K(h)\dd h \nonumber \\
	&\leq C\int_{2\leq \abs{h}\leq R} (\abs{h}^\gam - 1) K(h)\dd h + \int_{\abs{h}\geq R} 2\norm{\tilde u}_{L^\infty} K(h) \dd h,
	\label{e:al3}
\end{align}
where we note the use of the fact that $v-(\tilde u - \max \tilde u +1)\equiv0$ and also $\grad (v-(\tilde u - \max \tilde u +1)) \equiv 0 $ in $Q_2$.  Thus given any $\rho$, we can make $c(\gam)<\rho$ by first choosing $R$ large enough so that the tails of $K$ are negligible outside of $B_{R}$-- hence controlling the second term of (\ref{e:al3})-- and then choosing $\gam$ small enough so that second term of (\ref{e:al3}) is small enough.  Since none of these choices depend upon the kernel, $K$, they hold for $M^+$, and hence $c(\gam)$, as well.

Applying Theorem \ref{t:leps}, 
\begin{align*}
\min_{Q_{1/4}} v + \eps_0 + c(\gamma) &\geq \frac{ 1}{ C_6} \left( \int_{B_{1/4} \times [-1,-2^{-\alpha}]} v^\eps \dx \dd t \right)^{1/\eps}, \\
&\geq \frac{1}{C_6} \left( \frac 12 |B_{1/4}| (1-2^{-\alpha}) \right)^{1/\eps} \frac 12.
\end{align*}
Let us choose $\eps_0>0$ and $\gamma>0$ sufficiently small so that
\[ \delta := \frac{1}{C_6} \left( \frac 12 |B_{1/4}| (1-2^{-\alpha}) \right)^{1/\eps} \frac 12 - \eps_0 - c(\gamma) > 0.\]
Therefore, we obtained $\min_{Q_{1/4}} v \geq \delta$, which implies that $\osc_{Q_{1/4}} \tilde u \leq 1-\delta$. In term of the original variables, this means that 
\[ \osc_{Q_{8^{-k}}} u \leq 2 \times 8^{-\gamma (k-1)} (1-\delta).\]
Consequently, for any $r \in (8^{-k-1},8^{-k})$, 
\[ \osc_{Q_r} u \leq 2 \times 8^{-\gamma (k-1)} (1-\delta).\]
Choosing $\gamma$ sufficiently small so that
\[ 8^{-2\gamma} \geq (1-\delta),\]
implies that \eqref{e:holderaim} holds for all $r > 2^{-k-1}$. This finishes the inductive step, and hence the proof.

Note that there is no circular dependence between the constants $\gamma$ and $\eps_0$. All conditions required in the proof are satisfied for any smaller value. We choose $\eps_0$ and $\gamma$ sufficiently small so that all these conditions are met.
\end{proof}


\section{$C^{1,\gam}$ regularity for nonlinear equations} 

\label{s:nonlinear}

It is by now standard that a H\"older regularity result as in Theorem \ref{thm:intro} for kernels $K$ which have rough dependence in $x$ and $t$ implies a $C^{1,\alpha}$ estimate for solutions to nonlinear equations. The following is a more precise statement.

\begin{thm} \label{t:c1a}
Assume $\alpha_0>1$, $\alpha \in [\alpha_0,2]$ and $I$ is a translation invariant nonlocal operator which is uniformly elliptic with respect to the class of kernels that satisfy (A1), (A2), (A3) and (A4). Let $u: \R^n \times [-T,0] \to \R$ be a bounded viscosity solution of the following equation
\[ u_t - Iu = f \text{ in } B_1 \times [-T,0].\]
Then $u(\cdot,t) \in C^{1+\gamma}(B_{1/2})$ for all $t \in [-T/2,0]$ and $u(x,\cdot) \in C^{(1+\gamma)/2}([-T/2,0])$ for all $x \in B_{1/2}$. Moreover, the following regularity estimate holds,
\[ \sup_{t \in [-T/2,0]} \|u(\cdot,t)\|_{C^{1+\gamma}(B_{1/2})} + \sup_{x \in B_{1/2}} \|u(x,\cdot)\|_{C^{(1+\gamma)/2}([-T/2,0])} \leq C \left( \|u\|_{L^\infty(\R^n\times [-T,0])} + \|f\|_{L^\infty(B_1 \times [-T,0])} + I0 \right).\]
The constants $C$ and $\gamma$ depend only on $\lambda$, $\Lambda$, $\mu$, $n$ and $\alpha_0$. Here $\gamma>0$ is the minimum between $\alpha_0-1$ and the constant $\gamma$ from Theorem \ref{thm:intro} (or Theorem \ref{t:holder-with-drift}). 
\end{thm}

The proof of Theorem \ref{t:c1a} is given in \cite{serra2014regularity} for the smaller class of symmetric kernels satisfying \eqref{e:pointwise-bound-for-K}. The proof in \cite{serra2014regularity} uses the main result in \cite{lara2014regularity}. The proof of Theorem \ref{t:c1a} follows simply by replacing the use of the result of \cite{lara2014regularity}  in \cite{serra2014regularity} by Theorem \ref{t:holder-with-drift} in this paper. There is only one comment that needs to be made. In \cite{serra2014regularity}, the following quantity is used a few times to control the tail of an integral operator
\[ \|u\|_{L^1(\R^n,\omega_0)} := \int_{\R^n} u(x) (1+|x|)^{-n-\alpha_0} \dd x.\]
Because of our assumption \eqref{e:assumption-above}, this quantity is not sufficient and needs to be replaced by
\[ \max \left\{ x \in \R^n : (1+|x|)^{\eps-\alpha_0} u(x) \right\},\]
for some arbitrary small $\eps > 0$. After this small modification, the proof in \cite{serra2014regularity} straight forwardly applies to prove Theorem \ref{t:c1a} using Theorem \ref{t:holder-with-drift}.

The main example of a nonlinear integral operator $I$ is given by the Isaacs operator from stochastic games
\[ Iu(x) = \inf_i \sup_j \int_{\R^n} \delta_h u(x,t) K^{ij}(h) \dd h.\]
Here, the kernels $K^{ij}$ must satisfy the hypothesis (A1), (A2), (A3) and (A4) uniformly in $i$ and $j$.

The result can also be extended non translation invariant kernels $K^{ij}(x,h,t)$ provided that they are continuous with respect to $x$ and $t$. See \cite{serra2014regularity} for a discussion on this extension.


\appendix

\section{The crawling ink spots theorem}

In this section we prove a version of the \emph{crawling ink spots theorem} for fractional parabolic equations. This is a covering argument which first appeared in the original work of Krylov and Safonov \cite{krylov1979estimate}. In that paper it is indicated that the result was previously known by Landis, and it was Landis himself the one who came up with its suggestive name.

Let $d_\alpha$ be the parabolic distance of order $\alpha$. By definition, it is
\[ d_\alpha((x_0,t_0),(x_1,t_1)) = \max\left((2|t_1-t_2|)^{1/\alpha} , |x_1-x_2| \right).\]

The parabolic cylinders $Q_r(x,t)$ are balls of radius $r$ centered at $(x,t-r^\alpha/2)$ with respect to the distance $d_\alpha$. The importance of this characterization is that it allows us to use the Vitali covering lemma, since this result is valid in arbitrary metric spaces.

\begin{lemma} \label{l:inkspots}
Let $\mu>0$ and $E \subset F \subset B_1 \times \R$ be two open sets which satisfy the following two assumptions
\begin{itemize}
\item For every point $(x,t) \in F$, there exists a cylinder $Q \subset B_1 \times \R$ so that $(x,t) \in Q$ and $|E \cap Q| \leq (1-\mu) |Q|$.
\item For every cylinder $Q \subset B_1 \times \R$ such that $|E \cap Q| > (1-\mu) |Q|$, we have $Q \subset F$.
\end{itemize}
Then $|E| \leq (1-c\mu) |F|$, where $c$ is a constant depending on dimension only.
\end{lemma}

\begin{proof}
For every point $(x,t) \in F$, let $Q^0$ be the cylinder such that $(x,t) \in Q^0$ and $|E \cap Q^0| < (1-\mu) |Q^0|$.

Recall that $F$ is an open set. Let us choose a maximal cylinder $Q^{(x,t)}$ such that $(x,t) \in Q^{(x,t)}$, $Q^{(x,t)} \subset Q^0$ and $Q^{(x,t)} \subset F$. Two things may happen, either $Q^{(x,t)} = Q^0$, in which case $|Q^{(x,t)} \cap E| < (1-\mu) |Q^{(x,t)}|$ or for any larger cylinder $Q^{(x,t)} \subset Q \subset Q^0$ we would have $Q \not \subset F$. In the latter case we would have $|E \cap Q| \leq (1-\mu)|Q|$ for \textbf{any} cylinder $Q$ so that $Q^{(x,t)} \subset Q \subset Q^0$. In particular, the inequality holds for a decreasing sequence converging to $Q^{(x,t)}$ and therefore $|E \cap Q^{(x,t)} | \leq (1-\mu) |Q^{(x,t)}|$.

In any case, we have constructed a cover $Q^{(x,t)}$ of the set $F$ so that for all $(x,t) \in F$,
\begin{itemize}
\item $(x,t) \in Q^{(x,t)}$.
\item $Q^{(x,t)} \subset F$.
\item $|Q^{(x,t)} \cap E| \leq (1-\mu) |Q^{(x,t)}|$.
\end{itemize}

Using the Vitali covering lemma, we can select a countable subcollection of cylinders $Q_j$ such that $F \subset \bigcup_{j=1}^\infty 5Q_j$. Here each $Q_j$ is one of the cylinders $Q^{(x,t)}$. We write $5Q_j$ to denote the cylinder expanded as a ball with respect to the metric $d_\alpha$ with the same center and five times the radius.

Since $Q_j \subset F$ and $|E \cap Q_j| \leq (1-\mu) |Q_j|$, then $|Q_j \cap (F\setminus E)| \geq \mu |Q_j|$. Therefore,
\begin{align*}
|F \setminus E| &\geq \sum_{j=1}^\infty |Q_j \cap (F \setminus E)|, \\
&\geq \sum_{j=1}^\infty \mu |Q_j|,\\
= 5^{-d-\alpha} \mu \sum_{j=1}^\infty |5 Q_j| &\geq 5^{-d-\alpha} \mu |F|.
\end{align*}
The lemma follows with $c = 5^{-d-\alpha}$. 
\end{proof}

Lemma \ref{l:inkspots} is not applicable directly for parabolic equations. What we need is a covering lemma so that if $|E \cap Q| \geq (1-\mu) |Q|$, then a time-shift of the cylinder $Q$ is included in $F$ instead of $Q$ itself. This time-shift is given by the cylinders $\bar Q^m$ which we defined in section \ref{s:Leps}.

We now give the proof of the \emph{crawling ink spots} theorem.

\begin{proof}[Proof of Theorem \ref{t:ink-spots}]
Let $\mathcal Q$ be the collection of cylinders $Q \subset B_1 \times \R$ such that $|E \cap Q| > (1-\mu) |Q|$. Let $G = \bigcup_{Q \in \mathcal Q} Q$. By construction, $E$ and $G$ satisfy the assumptions of Lemma \ref{l:inkspots}, thus $|E| \leq (1-c\mu) |G|$. In order to prove this theorem, we are left to show that $|G| \leq (m+1)/m |F|$. For that, we will see that
\begin{align*}
\left\vert \bigcup_{Q \in \mathcal Q} \bar Q^m \right\vert &\geq \frac {m}{m+1} \left\vert \bigcup_{Q \in \mathcal Q} Q \cup \bar Q^m \right\vert \\
&\geq \frac {m}{m+1} |G|.
\end{align*}
The second inequality above is trivial by the inclusion of the sets. The first inequality is not obvious since the cylinders may overlap. We justify this first inequality below.

From Fubini's theorem, the measure of any set $A \in B_1 \times \R$ is given by
\[ |A| = \int_{B_1} \mathcal L_1( A \cap (\{x\} \times \R) ) \dd x,\]
where $\mathcal L_1$ stands for the one dimensional Lebesgue measure.

We finish the proof applying Fubini's theorem and noticing that for all $x \in B_1$,
\[ \mathcal L_1 \left( \bigcup_{Q \in \mathcal Q} \bar Q^m \cap (\{x\} \times \R)  \right) \geq \frac {m}{m+1} \mathcal L_1 \left(\bigcup_{Q \in \mathcal Q} \left( Q \cup \bar Q^m\right) \cap (\{x\} \times \R)  \right)
\]
This inequality follows from Lemma \ref{lem:cover-1d} which is described below.
\end{proof}

The following lemma is copied directly from Lemma 2.4.25 in \cite{imbert2013introduction}. An elementary proof is given there, which is independent of the rest of the text.

\begin{lemma}\label{lem:cover-1d}
Consider two (possibly infinite) sequences of real numbers
$(a_k)_{k=1}^N$ and $(h_k)_{k=1}^N$ for $N \in \mathbb N \cup \{\infty\}$ with
$h_k >0$ for $k =1,\dots,N$. Then
\[
\left|\cup_{k=1}^N (a_k, a_k + (m+1)h_k) \right| \le \frac{m}{m+1}
\left|\cup_{k=1}^N (a_k+h_k, a_k+ (m+1)h_k)\right|.
\]
\end{lemma}

\bibliographystyle{plain}
\bibliography{pie}
\index{Bibliography@\emph{Bibliography}}%

\end{document}